\documentclass[12pt]{amsart}

\usepackage{amssymb, amsmath, amsthm, mathrsfs, longtable, wrapfig, multicol, wasysym, subcaption, booktabs, cancel, comment}
\usepackage[margin=1.0in]{geometry}

\usepackage{latexsym}
\usepackage{hyperref}
\usepackage{enumerate}

\newtheorem{questionIntro}{Question}

\newtheorem{theorem}{Theorem}[section]

\newtheorem{proposition}[theorem]{Proposition}

\newtheorem{corollary}[theorem]{Corollary}

\theoremstyle{remark}

\numberwithin{equation}{section}

\newcommand{\calU}{\ensuremath{\mathcal{U}}}
\newcommand{\calV}{\ensuremath{\mathcal{V}}}
\newcommand{\calS}{\ensuremath{\mathcal{S}}}
\newcommand{\calT}{\ensuremath{\mathcal{T}}}

\newcommand{\calI}{\ensuremath{\mathcal{I}}}

\newcommand{\gln}{\ensuremath{\operatorname{GL}}}

\newcommand{\modu}{\ensuremath{\mathrm{\;mod\;}}}

\newcommand{\scrS}{\ensuremath{\mathscr{S}}}
\newcommand{\scrT}{\ensuremath{\mathscr{T}}}
\newcommand{\scrV}{\ensuremath{\mathscr{V}}}

\newcommand{\bbZ}{\ensuremath{\mathbb{Z}}}

\begin{document}

\title{Infinite co-minimal pairs in the integers and integral lattices}

\author{Arindam Biswas}
\address{Department of Mathematics, Technion - Israel Institute of Technology, Haifa 32000, Israel}
\curraddr{}
\email{biswas@campus.technion.ac.il}
\thanks{}

\author{Jyoti Prakash Saha}
\address{Department of Mathematics, Indian Institute of Science Education and Research Bhopal, Bhopal Bypass Road, Bhauri, Bhopal 462066, Madhya Pradesh,
India}
\curraddr{}
\email{jpsaha@iiserb.ac.in}
\thanks{}



\begin{abstract}
Given two nonempty subsets $A, B$ of a group $G$, they are said to form a co-minimal pair if $A \cdot B = G$, and $A' \cdot B \subsetneq G$ for any $\emptyset \neq A' \subsetneq A$ and $A\cdot B' \subsetneq G$ for any $\emptyset \neq B' \subsetneq B$. In this article, we show several new results on co-minimal pairs in the integers and the integral lattices. We prove that for any $d\geq 1$, the group $\bbZ^{2d}$ admits infinitely many automorphisms such that for each such automorphism $\sigma$, there exists a subset $A$ of $\bbZ^{2d}$ such that $A$ and $\sigma(A)$ form a co-minimal pair.  The existence and construction of co-minimal pairs in the integers with both the subsets $A$ and $B$ ($A\neq B$) of infinite cardinality was unknown. We show that such pairs exist and explicitly construct these pairs satisfying a number of algebraic properties.
\end{abstract}

\maketitle

\section{Introduction}


 Let $(G,+)$ be an abelian group and $W\subseteq G$ be a nonempty subset. A nonempty set $W'\subseteq G$ is said to be an \textit{additive complement} to $W$ if 
$W + W' = G.$ Additive complements have been studied since a long time in the context of representations of the integers, e.g., they appear in the works of Erd\H{o}s, Hanani, Lorentz and others. See \cite{Lorentz54}, \cite{Erdos54}, \cite{ErdosSomeUnsolved57} etc. 
In \cite{NathansonAddNT4}, Nathanson introduced the notion of \emph{minimal additive complements} for nonempty subsets of groups. An additive complement $W'$ to $W$ is said to be minimal if no proper subset of $W'$ is an additive complement to $W$, i.e., 
$$W + W' = G \,\text{ and }\, W + (W'\setminus \lbrace w'\rbrace)\subsetneq G \,\,\, \forall w'\in W'.$$
Nathanson was interested in a number of questions from metric geometry arising in the context of discrete groups. As an example, the existence of minimal nets in groups is strongly related to the existence of minimal additive complements of generating sets, see \cite[\S 1]{NathansonAddNT4} (see also \cite[Problem 1]{NathansonAddNT4}). A notion stronger to minimal additive complements is that of \emph{additive co-minimal pairs}. Given two nonempty subsets $A, B$ of a group $G$, they are said to form a co-minimal pair if $A \cdot B = G$, and $A' \cdot B \subsetneq G$ for any $\emptyset \neq A' \subsetneq A$ and $A\cdot B' \subsetneq G$ for any $\emptyset \neq B' \subsetneq B$. Thus, they are pairs $(A\in G,B\in G)$ such that each element in a pair is a minimal additive complement to the other. The notion of co-minimal pairs was considered in a prior work of the authors see \cite[Definition 1.2]{CoMin1}. Henceforth, by a complement we shall mean an additive complement. If we mean set-theoretic complement, we shall explicitly state it. 

It is a challenging task to classify the co-minimal pairs in a given group. Even in the context of the group of integers $\bbZ$, they are not completely understood. In \cite{CoMin1}, it was shown that non-empty finite subsets in free abelian groups (not necessarily of finite rank) belong to co-minimal pairs. However, the existence and the construction of infinite co-minimal pairs in $\mathbb{Z}$, i.e., co-minimal pairs $(A,B)$ where both $A$ and $B$ are infinite ($A\neq B$) has been unknown. One of our motivations in this article is to show the existence and give explicit constructions of these pairs. Further, we also study co-minimal pairs in higher rank integer lattices. Our constructions satisfy certain nice combinatorial and group theoretic properties. They are mainly motivated by the following questions:

\begin{questionIntro}
	\label{Qn:BddBelowBddAbove}
	Does there exist infinite subsets $A, B$ of $\bbZ$ which form a co-minimal pair and one of them is bounded below and the other is bounded above?
\end{questionIntro}

\begin{questionIntro}
	\label{Qn:OneIsSymm}
	Does there exist infinite subsets $A, B$ of $\bbZ$ which form a co-minimal pair and at least one of them is a symmetric subset of $\bbZ$?
\end{questionIntro}

In \cite{Kwon} and \cite{CoMin1}, it was established that $(A,A)$ is a co-minimal pair in an abelian group $G$, if and only if $A + A = G$ and $A$ avoids $3$-term arithmetic progressions. This motivates the following two questions when we note that the trivial automorphism $\sigma(g) = g, \forall g\in G$ fixes any subset $A\in G$.

\begin{questionIntro}
\label{Qn:AExists}
Given an automorphism $\sigma$ of a group $G$, does there exist subsets $A$ in $G$ such that $(A, \sigma (A))$ is a co-minimal pair? 
\end{questionIntro}

\begin{questionIntro}
\label{Qn:AExistsQuadrant}
Given an automorphism $\sigma$ of $\bbZ^d$, does there exist subsets $A$ in $G$ such that $A$ is contained in a quadrant
\footnote{
A subset $X$ of $\bbZ^d$ is said to be \textit{contained in a quadrant} if for any given $1\leq i \leq d$, the $i$-th coordinate of all the points of $X$ is either positive or negative. 
}
and $(A, \sigma (A))$ is a co-minimal pair? 
\end{questionIntro}

Note that Questions \ref{Qn:AExists}, \ref{Qn:AExistsQuadrant} are related to Questions \ref{Qn:OneIsSymm}, \ref{Qn:BddBelowBddAbove}. Indeed, having affirmative answers to Questions \ref{Qn:OneIsSymm}, \ref{Qn:BddBelowBddAbove} allows to answer Questions \ref{Qn:AExists}, \ref{Qn:AExistsQuadrant} for certain automorphisms of free abelian groups. For instance, if $\calU, \calV$ are infinite subsets of $\bbZ$ forming a co-minimal pair and $\calV$ is symmetric, then taking $A  = \calU \times \calV, \calV \times \calU$, we obtain an affirmative answer to Question \ref{Qn:AExists} when $G = \bbZ^2$ and $\sigma
=
\left(
\begin{smallmatrix} 
0 & 1\\
-1 & 0 
\end{smallmatrix}
\right), 
\left(
\begin{smallmatrix} 
0 & -1\\
1 & 0 
\end{smallmatrix}
\right)$.
Moreover, if $\calS, \calT$ are infinite subsets of $\bbZ$ forming a co-minimal pair and $\calS$ (resp. $\calT$) is bounded above (resp. below) by $0$, then taking $A  = (-\calS) \times \calT$, we obtain an affirmative answer to Question \ref{Qn:AExistsQuadrant} when $d = 2$ and $\sigma
=
\left(
\begin{smallmatrix} 
0 & -1\\
-1 & 0 
\end{smallmatrix}
\right)$.

In the following, for $?\in \{<, \leq, > , \geq\}$ and $x\in \bbZ$, the set $\{n\in \bbZ\,|\, n?x\}$ is denoted by $\bbZ_{?x}$. 
Let $X$ be a subset of $\bbZ$ of the form $\{n \in \bbZ
\,|\, a\leq n \leq b\}$ for some $a, b\in \bbZ$. If $X$ contains an even number of elements (i.e., if $b-a + 1$ is even), then the subset $\{x\in X\,|\, x\leq \frac{b + a -1}{2}\}$ (resp. $\{x\in X\,|\, x\geq \frac{b + a +1}{2}\}$) of $X$ is called the left half (resp. right half) of $X$. 
If $b-a + 1$ is divisible by $4$, then the subsets 
$\{x\in X\,|\, x\leq a - 1 + \frac{b - a + 1}{4}\}$, 
$\{x\in X\,|\, a - 1 + \frac{b - a + 1}{4} < x \leq a - 1 + \frac{b - a + 1}{2}\}$,
$\{x\in X\,|\, a - 1 + \frac{b - a + 1}{2} < x \leq a - 1 + \frac{3(b - a + 1)}{4}\}$,
$\{x\in X\,|\, a - 1 + \frac{3(b - a + 1)}{4} < x \}$
of $X$ are called the first quarter, the second quarter, the third quarter, and the fourth quarter of $X$ respectively. The first quarter (resp. the fourth quarter) of $X$ is also called the left quarter (resp. the right quarter) of $X$.

\subsection{Statement of results}
In Theorems \ref{Thm:CoMinST}, \ref{Thm:CoMinUV}, we prove that Questions 
\ref{Qn:BddBelowBddAbove}, \ref{Qn:OneIsSymm} admit answers in the affirmative. 
Theorems \ref{Thm:CoMinST}, \ref{Thm:CoMinUV} follow from Theorems \ref{Thm:CoMinS}, \ref{Thm:CoMinU}. 

\begin{theorem}
\label{Thm:CoMinST}
Let $\calT$ denote the subset $\{1, 2, 2^2, \cdots\}$ of $\bbZ$. Then there exists an infinite subset $\calS$ of $\bbZ_{\leq -1}$ such that $(\calS, \calT)$ is a co-minimal pair. 
\end{theorem}

\begin{theorem}
\label{Thm:CoMinUV}
Let $\calV$ denote the subset of $\bbZ$ defined as 
$$\calV : = \{1, 2, 2^2, 2^3, \cdots\} \cup \{-1, -2, -2^2, -2^3, \cdots\}.$$
Then there exists an infinite subset $\calU$ of $\bbZ_{\leq -1}$ such that $(\calU, \calV)$ is a co-minimal pair. 
\end{theorem}

Using these two results, we establish Theorems \ref{Thm:AExists}, \ref{Thm:AExistsQuadrant}, which prove that Questions \ref{Qn:AExists}, \ref{Qn:AExistsQuadrant} admit answers in the affirmative for an infinite class of automorphisms of free abelian groups. An immediate consequence of Theorem \ref{Thm:AExists} is stated below. 

\begin{theorem}
\label{Thm:Z2d}
If $\sigma$ is an automorphism of $\bbZ^2$, i.e., an element of $
\gln_2(\bbZ)$ having exactly two nonzero entries, then there exists a subset $A$ of $\bbZ^2$ such that $(A, \sigma(A))$ forms a co-minimal pair in $\bbZ^2$. For any $d\geq 1$, the group $\bbZ^{2d}$ admits infinitely many automorphisms such that for each such automorphism $\sigma$, there exists a subset $A$ of $\bbZ^{2d}$ such that $A$ and $\sigma(A)$ form a co-minimal pair. 
\end{theorem}

\section{A co-minimal pair involving a bounded below subset}

In this section, we establish Theorem \ref{Thm:CoMinST}, which follows from Theorem \ref{Thm:CoMinS}.
Consider the subsets $\{J_n\}_{n\geq 0}, \{K_n\}_{n\geq 0}, \{I_n\}_{n\geq 0}, \{\calI_n\}_{n\geq 0}, S$ of $\bbZ$ defined by 
\begin{align*}
J_n 
& = 
\begin{cases}
\{1\} & \text{ if } n = 0, \\
\{1\} & \text{ if } n = 1, \\
\{1, 2, \cdots, 2^{n-2}\} \cup (2^{n-1} + J_{n-1}) & \text{ if } n\geq 2,
\end{cases}
\\
K_n 
& = 
\begin{cases}
J_0 & \text{ if } n = 0, \\
J_1 & \text{ if } n = 1, \\
J_2 \setminus\{1\} & \text{ if } n = 2, \\
\{2^{n-3}+1, 2^{n-3}+2, 2^{n-3}+3, \cdots, 2^{n-2}\} \cup (2^{n-1} + 2^{n-2} + J_{n-2}) &  \text{ if }n\geq 3,
\end{cases}
\\
I_n 
& = K_n - (1 + 2^{n+1})\quad \text{ if } n\geq 0,
\\
\calI_n 
& = \begin{cases}
\{-2, -1\} & \text{ if } n = 0, \\
\{1, 2, 3, \cdots, 2^n\} - (1+2^{n+1}) & \text{ if } n \geq 1,
\end{cases}
\\
S 
& = \cup _{n\geq 0} I_n.
\end{align*}

\begin{proposition}
The set $\calT$ is an additive complement of $S$ in $\bbZ$. 
\end{proposition}

\begin{proof}
For any $m\geq 3$, 
$$\cup_{k = 3}^m (I_k + 2^{k+1})
\supseteq 
\cup_{k = 3}^m \{2^{k-3}, \cdots, 2^{k-2}-1\} = \{1, 2, \cdots, 2^{m-2}-1\}
$$
holds, 
which implies that $S + \calT$ contains $\bbZ_{\geq 1}$.
Note that $I_n$ is contained in $\calI_n$ for all $n\geq 0$, and the union of the sets $\{\calI_n\}_{n\geq 0}$ is equal to $\bbZ_{\leq -1}$ and these sets lie next to each other in the sense that $\min \calI_n = 1 + \max \calI_{n+1}$ for all $n\geq 0$. Thus, to prove that $S+\calT$ contains $\bbZ$, it suffices to show that $S + \calT$ contains $\{0\}$ and $\calI_n$ for all $n\geq 0$, which we establish by proving the following statements. 
\begin{enumerate}
\item $S + \calT$ contains the left half of $\calI_n$ for all $n\geq 1$, 
\item $S + \calT$ contains the left quarter of the right half of $\calI_n$ for all $n\geq 3$, 
\item $S + \calT$ contains the second quarter of the right half of $\calI_n$ for all $n\geq 3$, 
\item $S + \calT$ contains the right quarter of $\calI_n$ for all $n\geq 4$, 
\item $S + \calT$ contains $\{0\},\calI_0, \calI_1$, the right half of $\calI_2$ and the right quarter of $\calI_3$.
\end{enumerate}

Note that the sets $I_0 + \{1, 2\}$, $I_1 + \{1, 2\}$, $I_2 + \{1, 2\}$, $I_3 + 2^2$ contain $\{-1, 0\}$, $\{-3, -2\}$, $\{-5, -4\}$, $\{-6\}$ respectively. So, $S + \calT$ contains $\calI_1 \cup \calI_0 \cup \{0\}$ and the right half of $\calI_2$. Since $I_3 + 1$ contains $\{-9\}$ and $I_4 + 2^3$ contains $\{-10\}$, the set $S + \calT$ contains the right quarter of $\calI_3$. This establishes the fifth statement. 

Note that for any $n\geq 1$, the left half of $\calI_n$ is contained in $I_{n+3}+2^{n+3}$, i.e., the inclusions 
\begin{align*}
I_{n+3} + 2^{n+3} 
& = K_{n+3} - ( 1 + 2^{n+4}) + 2^{n+3} \\
&= K_{n+3} - (1 + 2^{n+3})\\
& \supseteq (2^{n+2} + 2^{n+1} + J_{n+1}) - (1 + 2^{n+3})\\
&= J_{n+1}- (1 + 2^{n+1})\\
& \supseteq \{1, 2, 3, \cdots, 2^{n-1}\}  - (1 + 2^{n+1})
\end{align*}
hold, thus establishing the first statement. 

Note that for any $n\geq 3$, the left quarter of the right half of $\calI_n$ is contained in $I_{n+1} + 2^{n}$, i.e., the inclusions
\begin{align*}
I_{n+1} + 2^{n} 
& = K_{n+1} - ( 1 + 2^{n+2}) + 2^{n} \\
&= K_{n+1} - (1 + 2^{n+2}) + 2^n\\
& \supseteq (2^{n} + 2^{n-1} + J_{n-1}) - (1 + 2^{n+2})+2^n\\
&= (2^{n-1} + J_{n-1})- (1 + 2^{n+2}) + 2^{n+1}\\
&= (2^{n-1} + J_{n-1})- (1 + 2^{n+1}) \\
& \supseteq \{2^{n-1}+1, 2^{n-1}+2, 2^{n-1}+3, \cdots, 2^{n-1}+2^{n-3}\}  - (1 + 2^{n+1})
\end{align*}
hold, 
the second quarter of the right half of $\calI_n$ is contained in $I_n + 2^{n-1}$, i.e., the inclusions
\begin{align*}
I_n + 2^{n-1}
& = K_n - (1 + 2^{n+1}) + 2^{n-1} \\
& \supseteq \{2^{n-3}+1, 2^{n-3}+2, 2^{n-3}+3, \cdots, 2^{n-2}\}  - (1 + 2^{n+1}) + 2^{n-1} \\
& = \{2^{n-1}+2^{n-3}+1, 2^{n-1}+2^{n-3}+2, 2^{n-1}+2^{n-3}+3, \cdots, 2^{n-1}+2^{n-2}\}  - (1 + 2^{n+1}) 
\end{align*}
hold. This establishes the second and the third statement.

For $n\geq 3$, the points in the right quarter of $\calI_n$ that lie $I_n$ are contained in $I_{n+1}+2^n$, i.e., the inclusions
\begin{align*}
I_{n+1} + 2^{n} 
&= (2^{n-1} + J_{n-1})- (1 + 2^{n+1}) \\
& \supseteq (2^{n-1}+2^{n-2}+J_{n-2})  - (1 + 2^{n+1})
\end{align*}
hold. Note that for any $m\geq 2$, 
$$J_m + \{0, 1, 2, 2^2, \cdots, 2^{m-2}\} \supseteq \{1, 2, 3, \cdots, 2^m\}$$
holds. Indeed, it holds for $m = 2$, and assuming that 
$$J_k + \{0, 1, 2, 2^2, \cdots, 2^{k-2}\} \supseteq \{1, 2, 3, \cdots, 2^k\}$$
holds for some integer $k\geq 2$, it follows that the inclusions 
\begin{align*}
& J_{k+1} + \{0, 1, 2, 2^2, \cdots, 2^{k-2}, 2^{k-1}\} \\
& \supseteq (\{1, 2, \cdots, 2^{k-1}\} \cup (2^{k} + J_{k})) + \{0, 1, 2, 2^2, \cdots, 2^{k-2}, 2^{k-1}\} \\
& \supseteq (\{1, 2, \cdots, 2^{k-1}\}  + \{0, 2^{k-1}\}) \cup ((2^{k} + J_{k}) + \{0, 1, 2, 2^2, \cdots, 2^{k-2}\}) \\
& = \{1, 2, \cdots, 2^{k}\} \cup (2^{k} + (J_{k} + \{0, 1, 2, 2^2, \cdots, 2^{k-2}\})) \\
& \supseteq \{1, 2, \cdots, 2^{k}\} \cup (2^{k} + \{1, 2, 3, \cdots, 2^k\}) \\
& = \{1, 2, \cdots, 2^{k}\} \cup  \{2^{k} +1,2^{k} + 2, 2^{k} +3, \cdots, 2^{k+1}\} \\
& = \{1, 2, \cdots, 2^{k+1}\} 
\end{align*}
hold. Consequently, for any $n\geq 4$, the points in the right quarter of $\calI_n$ are contained in $(I_{n+1} + 2^n) \cup ( I_n  + \{1, 2, 2^2, \cdots, 2^{n-4}\})$, i.e., the inclusions
\begin{align*}
& (I_{n+1} + 2^n) \cup ( I_n  + \{1, 2, 2^2, \cdots, 2^{n-4}\}) \\
& \supseteq 
((2^{n-1}+2^{n-2}+J_{n-2})  - (1 + 2^{n+1})) 
\cup 
(((2^{n-1}+2^{n-2}+J_{n-2})  - (1 + 2^{n+1}))+ \{1, 2, 2^2, \cdots, 2^{n-4}\})\\
& = ((2^{n-1}+2^{n-2}+J_{n-2})  - (1 + 2^{n+1}))+ \{0, 1, 2, 2^2, \cdots, 2^{n-4}\}\\
& = (2^{n-1}+2^{n-2}+(J_{n-2}+ \{0, 1, 2, 2^2, \cdots, 2^{n-4}\}))  - (1 + 2^{n+1})\\
& \supseteq (2^{n-1}+2^{n-2}+\{1, 2, 3, \cdots, 2^{n-2}\})  - (1 + 2^{n+1})\\
& = \{2^{n-1}+2^{n-2}+1, 2^{n-1}+2^{n-2}+2, 2^{n-1}+2^{n-2}+3, \cdots, 2^{n}\}  - (1 + 2^{n+1})
\end{align*}
hold. This proves the fourth statement. So, the set $S + \calT$ contains $\bbZ$. 
\end{proof}

In the above proof, we established that for any $n\geq 3$, the left quarter of the right half of $\calI_n$ is contained in $I_{n+1} + 2^{n}$, and the second quarter of the right half of $\calI_n$ is contained in $I_n + 2^{n-1}$. In fact, the points of $I_{n+1}$ (which are precisely the points in the left quarter of the right quarter of $\calI_{n+1}$) that yield the left quarter of the right half of $\calI_n$ and the points of $I_n$ (which are precisely the points in the second quarter of the left half of $\calI_n$) that yield the second quarter of the right half of $\calI_n$ cannot be removed from $S$ in a  sense made precise in the following result. 

\begin{theorem}
\label{Thm:RequiredClusters}
Let $\scrS, \scrT$ be nonempty subsets of $S, \calT$ respectively such that $\scrS + \scrT = \bbZ$. Then for $n\geq 3$, the set $\scrT$ contains $2^{n-1}$, $\scrS$ contains the points in the second quarter of the left half of $\calI_n$, i.e.,
\begin{equation}
\label{Eqn:LeftHalf}
\scrS \supseteq \{2^{n-3} + 1, 2^{n-3} + 2, 2^{n-3} + 3, \cdots, 2^{n-2}\} - (1 + 2^{n+1})
\quad 
\text{ for }
n\geq 3,
\end{equation}
and the points in the left quarter of the right quarter of $\calI_{n+1}$, i.e.,  
\begin{equation}
\label{Eqn:LeftQuarDeRightQuar}
\scrS \supseteq 
\{2^{n}+ 2^{n-1} + 1, 2^{n}+ 2^{n-1} + 2, 2^{n}+ 2^{n-1} + 3, \cdots, 2^{n}+ 2^{n-1} + 2^{n-3}\}  - (1 + 2^{n+2})
\quad 
\text{ for }
n\geq 3.
\end{equation}
Moreover, $\scrT$ contains $1, 2$, $\scrS$ contains $-2, -4$. Consequently, $\calT$ is a minimal complement of $S$. 
\end{theorem}

\begin{proof}
We claim that for any $n\geq 2$, no point in the right half of $\calI_n$ lie in $\cup _{m \neq n, n+1}I_m + \calT$, i.e., 
\begin{equation}
\label{Eqn:RightHalfNotObtained}
\left(
\{2^{n-1}+1, 2^{n-1}+2, 2^{n-1}+3, \cdots, 2^n\} - (1 + 2^{n+1})
\right) 
\cap 
\left(
\cup _{m \neq n, n+1}I_m + \calT
\right)
= \emptyset
\quad
\text{ for } 
n\geq 2.
\end{equation}
Since the inclusions
\begin{align*}
(\cup _{0 \leq m < n} I_m) + \calT
& \subseteq 
(\cup _{0 \leq m < n} \calI_m) + \calT\\
& \subseteq 
\{- 2^n,-2^n+1, -2^n+2, -2^n + 3, \cdots, -1\} + \calT
\end{align*}
hold and the elements of $\calT$ are positive, it follows that the smallest element of $(\cup _{0 \leq m < n} I_m) + \calT$ is greater than $-2^n$, which is greater than the largest element of $\{2^{n-1}+1, 2^{n-1}+2, 2^{n-1}+3, \cdots, 2^n\} - (1 + 2^{n+1})$, and hence 
\begin{equation}
\label{Eqn:RightHalfNotObtainedRHS}
\left(
\{2^{n-1}+1, 2^{n-1}+2, 2^{n-1}+3, \cdots, 2^n\} - (1 + 2^{n+1})
\right) 
\cap 
\left(
\cup _{0 \leq m < n}I_m + \calT
\right) 
= \emptyset
\quad
\text{ for } 
n\geq 1.
\end{equation}
For any $n\geq 2$, the inclusion $I_n \subseteq \calI_n = \{1, 2, 3, \cdots, 2^n\}- (1 + 2^{n+1})$ yields that $I_n \cap \bbZ_{\geq 1} = \emptyset, I_n \cap \bbZ_{\leq - 2^{n+1} - 1} = \emptyset$, and for any $m\geq n+2$, the inclusions
\begin{align*}
I_m + 2^k 
& \subseteq 
\calI_m + 2^k \\
& \subseteq 
(\{1, 2, 3, \cdots, 2^m\} - (1 + 2^{m+1})) + 2^k \\
& = 
\{1, 2, 3, \cdots, 2^m\}  + 2^k-  2^{m+1} - 1\\
& \subseteq \bbZ_{\geq 1}
\end{align*}
hold for any $k\geq m+1$, and the inclusions
\begin{align*}
I_m + 2^k 
& \subseteq 
\bbZ_{\leq 2^m - (1 + 2^{m+1})} + 2^k \\
& = \bbZ_{\leq 2^m - (1 + 2^{m+1}) + 2^k } \\
& = \bbZ_{\leq -1-2^m + 2^k} \\
& \subseteq 
\bbZ_{\leq  - 1 - 2^m + 2^{m-1}} \\
& \subseteq 
\bbZ_{\leq  - 1 - 2^{m-1}} \\
&\subseteq
\bbZ_{\leq - 2^{n+1} - 1}
\end{align*}
hold for $0 \leq k < m$, and hence 
\begin{equation}
\label{Eqn:RightHalfNotObtainedLHSSans1pt}
\left(
\{2^{n-1}+1, 2^{n-1}+2, 2^{n-1}+3, \cdots, 2^n\} - (1 + 2^{n+1})
\right) 
\cap 
\left(
\cup_{m \geq n+2} (I_m + (\calT \setminus \{2^m\}))
\right)
= \emptyset
\quad
\text{ for } 
n\geq 2.
\end{equation}
For any $m\geq 3$, the inclusions 
\begin{align*}
I_m + 2^m 
& \subseteq 
(\{2^{m-3}+1, 2^{m-3}+2, 2^{m-3}+3, \cdots, 2^{m-2}\} \cup (2^{m-1} + 2^{m-2} + J_{m-2})) - (1 + 2^{m+1}) + 2^m \\
& \subseteq 
(J_m \cup (2^{m-1} + 2^{m-2} + J_{m-2})) - (1 + 2^m) \\
&\subseteq 
(J_m \cup (2^{m-1} + J_{m-1})) - (1 + 2^m) \\
&= 
J_m  - (1 + 2^m) \\
& = 
(\{1, 2, \cdots, 2^{m-2}\} - (1 + 2^m))
\cup
\cdots 
\cup 
(\{1, 2, 3, 2^2\} - (1+2^4))
\cup
(\{1, 2\} - (1+2^3)) \\
& \qquad \cup 
(\{1\} - (1+2^2))
\cup 
(\{1\} - ( 1+ 2)) \\
& = 
\left(
\cup_{k=0}^{m-2} (\{1, 2, 3, \cdots, 2^k\} - (1 + 2^{k+2}))
\right) 
\cup 
(\{1\}-(1+2))
\end{align*}
hold, which implies that for $n\geq 2$ and $m\geq n+1$, the inclusions 
\begin{align*}
(I_m + 2^m) \cap \calI_n 
& = 
(I_m + 2^m) \cap (\{1, 2, 3, \cdots, 2^n\} - (1 + 2^{n+1}))\\
& \subseteq 
\left(
\left(
\cup_{k=0}^{m-2} (\{1, 2, 3, \cdots, 2^k\} - (1 + 2^{k+2}))
\right) 
\cup 
(\{1\}-(1+2))
\right) \cap \calI_n \\
& = 
\left(
\left(
\cup_{k=1}^{m-1} (\{1, 2, 3, \cdots, 2^{k-1}\} - (1 + 2^{k+1}))
\right) 
\cup 
(\{1\}-(1+2))
\right) \cap \calI_n \\
& = 
\{1, 2, 3, \cdots, 2^{n-1}\} - (1 + 2^{n+1}),
\end{align*}
hold, which yields
\begin{equation}
\label{Eqn:RightHalfNotObtainedLHS1pt}
\left(
\{2^{n-1}+1, 2^{n-1}+2, 2^{n-1}+3, \cdots, 2^n\} - (1 + 2^{n+1})
\right) 
\cap 
\left(
\cup_{m \geq n+1} (I_m + 2^m)
\right)
= \emptyset
\quad
\text{ for } 
n\geq 2.
\end{equation}
Consequently, Equation \eqref{Eqn:RightHalfNotObtained} follows from Equations \eqref{Eqn:RightHalfNotObtainedRHS}, \eqref{Eqn:RightHalfNotObtainedLHSSans1pt}, \eqref{Eqn:RightHalfNotObtainedLHS1pt}.

We claim that for any $n\geq 3$, no point in the second quarter of the right half of $\calI_n$ lie in $I_{n+1} + \calT$, i.e., 
\begin{equation}
\label{Eqn:2ndQuarDeRightHalfNotObtained}
\left(
\{
2^{n-1}+2^{n-3}+1, 2^{n-1}+2^{n-3}+2, 2^{n-1}+2^{n-3}+3,\cdots, 
2^{n-1} + 2^{n-2}
\} - (1 + 2^{n+1})
\right) 
\cap 
\left(
I_{n+1} + \calT
\right)
= \emptyset.
\end{equation}
This claim follows since for $n\geq 3$, the inclusions 
\begin{align*}
I_{n+1} + 2^k
& \subseteq 
\calI_{n+1} + 2^k\\
& \subseteq 
\bbZ_{\leq 2^{n+1} - (1 + 2^{n+2}) + 2^k} \\
& \subseteq 
\bbZ_{\leq 2^k- (1 + 2^{n+1})}\\
& \subseteq 
\bbZ_{\leq 2^{n-1}- (1 + 2^{n+1})}
\end{align*}
hold for $0 \leq k \leq n-1$, the inclusions 
\begin{align*}
I_{n+1} + 2^n
& = 
K_{n+1} - (1 + 2^{n+2}) + 2^n \\
& \subseteq 
\left(
\{2^{n-2}+1, 2^{n-2}+2, \cdots, 2^{n-1}\} \cup (2^n + 2^{n-1} + J_{n-1})
\right)
 - (1 + 2^{n+2}) + 2^n \\
& = 
\left(
\{2^{n-2}+1, 2^{n-2}+2, \cdots, 2^{n-1}\}  - (1 + 2^{n+2}) + 2^n 
\right)
\cup 
\left(
(2^n + 2^{n-1} + J_{n-1}) - (1 + 2^{n+2}) + 2^n 
\right) \\
& \subseteq 
\bbZ_{\leq 2^{n-1}  - (1 + 2^{n+2}) + 2^n}
\cup 
\left(
(2^n + 2^{n-1} + J_{n-1}) - (1 + 2^{n+2}) + 2^n 
\right) \\
& \subseteq 
\bbZ_{\leq 2^{n}  - (1 + 2^{n+2}) + 2^n}
\cup 
\left(
(2^n + 2^{n-1} + 
\left(
\{1, 2, \cdots, 2^{n-3}\}
\cup 
(2^{n-2} + J_{n-2})
\right)
) - (1 + 2^{n+2}) + 2^n 
\right) \\
& \subseteq 
\bbZ_{\leq  - (1 + 2^{n+1}) }
\cup 
\bbZ_{\leq 2^n + 2^{n-1} + 2^{n-3} - (1 + 2^{n+2}) + 2^n} 
\cup 
\bbZ_{\geq 2^n + 2^{n-1} + 2^{n-2} + 1 - (1 + 2^{n+2}) + 2^n}  \\
& \subseteq 
\bbZ_{\leq 2^{n}  - (1 + 2^{n+2}) + 2^n}
\cup 
\left(
(2^n + 2^{n-1} + 
\left(
\{1, 2, \cdots, 2^{n-3}\}
\cup 
(2^{n-2} + J_{n-2})
\right)
) - (1 + 2^{n+2}) + 2^n 
\right) \\
& \subseteq 
\bbZ_{\leq  - (1 + 2^{n+1}) }
\cup 
\bbZ_{\leq 2^{n-1} + 2^{n-3} - (1 + 2^{n+1})}
\cup 
\bbZ_{\geq 2^{n-1} + 2^{n-2} + 1 - (1 + 2^{n+1})}
\end{align*}
hold, 
and the inclusions 
\begin{align*}
I_{n+1} + 2^k 
& \subseteq 
\calI_{n+1} + 2^k \\
& \subseteq 
\bbZ_{\geq 1 - (1 + 2^{n+2}) + 2^k}\\
& = 
\bbZ_{\geq 2^k - 2^{n+2}}\\
& \subseteq 
\bbZ_{\geq 0}
\end{align*}
hold for $k \geq n+2$. 

We claim that for any $n\geq 3$, no point in the second quarter of the right half of $\calI_n$ lie in $I_n + (\calT\setminus \{2^{n-1}\})$, i.e., 
\begin{equation}
\label{Eqn:2ndQuarDeRightHalfNotObtainedIn}
\left(
\{
2^{n-1}+2^{n-3}+1, 2^{n-1}+2^{n-3}+2, \cdots, 
2^{n-1} + 2^{n-2}
\} - (1 + 2^{n+1})
\right) 
\cap 
\left(
I_n + (\calT\setminus \{2^{n-1}\})
\right)
= \emptyset.
\end{equation}
This claim follows since for $n\geq 3$, 
$$I_n = 
(\{
2^{n-3}+1, 2^{n-3}+2, 2^{n-3}+3, \cdots, 2^{n-2}
\} - (1 + 2^{n+1}))
\cup 
((2^{n-1}+ 2^{n-2} + J_{n-2}) - (1 + 2^{n+1})),
$$
the inclusions 
\begin{align*}
((2^{n-1}+ 2^{n-2} + J_{n-2}) - (1 + 2^{n+1})) + \scrT
& \subseteq 
\bbZ_{\geq 2^{n-1}+2^{n-2}+ 1 - (1 + 2^{n+1})} + \bbZ_{\geq 1}\\
& \subseteq 
\bbZ_{\geq 2^{n-1}+2^{n-2}+ 1 - (1 + 2^{n+1})} 
\end{align*}
hold, the inclusions 
\begin{align*}
& 
(\{2^{n-3}+1, 2^{n-3}+2, 2^{n-3}+3, \cdots, 2^{n-2}\}   - (1 + 2^{n+1}))+ 2^k 
\\
& \subseteq 
\bbZ_{\geq 2^{n-3} + 1 - (1 + 2^{n+1}) + 2^k}\\
& \subseteq 
\bbZ_{\geq 2^{n-3} + 1 - (1 + 2^{n+1}) + 2^n}\\
& = 
\bbZ_{\geq 2^{n-3} + 1 - (1 + 2^{n+1}) + 2^{n-1} + 2^{n-2} + 2^{n-2} }\\
& \subseteq 
\bbZ_{\geq 2^{n-1} + 2^{n-2} + 1 - (1 + 2^{n+1})} 
\end{align*}
hold for $k\geq n$, and the inclusions 
\begin{align*}
& 
(\{2^{n-3}+1, 2^{n-3}+2, 2^{n-3}+3, \cdots, 2^{n-2}\}   - (1 + 2^{n+1}))+ 2^k 
\\
& \subseteq 
\bbZ_{\leq 2^{n-2} - (1 + 2^{n+1}) + 2^k}\\
& \subseteq 
\bbZ_{\leq 2^{n-2} - (1 + 2^{n+1}) + 2^{n-2}}\\
& = 
\bbZ_{\leq 2^{n-1} - (1 + 2^{n+1}) }\\
& \subseteq 
\bbZ_{\leq 2^{n-1} + 2^{n-3} - (1 + 2^{n+1}) }
\end{align*}
hold for $k\leq  n-2$.

We now show that Equation \eqref{Eqn:LeftHalf} holds. From Equations \eqref{Eqn:RightHalfNotObtained}, \eqref{Eqn:2ndQuarDeRightHalfNotObtained}, 
it follows that 
$$(\scrS \cap I_n) + \scrT 
\supseteq 
\{
2^{n-1}+2^{n-3}+1, 2^{n-1}+2^{n-3}+2, 2^{n-1}+2^{n-3}+3,\cdots, 
2^{n-1} + 2^{n-2}
\} - (1 + 2^{n+1}).$$
Using Equation \eqref{Eqn:2ndQuarDeRightHalfNotObtainedIn}, it follows that $2^{n-1} \in \scrT$ and 
$$(\scrS \cap I_n) + 2^{n-1} 
\supseteq 
\{
2^{n-1}+2^{n-3}+1, 2^{n-1}+2^{n-3}+2, 2^{n-1}+2^{n-3}+3,\cdots, 
2^{n-1} + 2^{n-2}
\} - (1 + 2^{n+1}).$$
Thus Equation \eqref{Eqn:LeftHalf} holds. 

We claim that for any $n\geq 3$, no point in the left quarter of the right half of $\calI_n$ lie in $I_n+ \calT$, i.e., 
\begin{equation}
\label{Eqn:1stQuarDeRightHalfNotObtainedIn}
\left(
\{
2^{n-1} + 1, 2^{n-2} + 2, \cdots, 2^{n-1}+2^{n-3}
\} - (1 + 2^{n+1})
\right) 
\cap 
\left(
I_n + \calT
\right)
= \emptyset
\quad
\text{ for } 
n\geq 3.
\end{equation}
This claim follows since for $n\geq 3$, the inclusions 
\begin{align*}
I_n + 2^k 
& \subseteq 
K_n - (1 + 2^{n+1}) + 2^k \\
& \subseteq 
\left( 
\{2^{n-3} + 1, 2^{n-3} + 2, \cdots, 2^{n-2} \} \cup (2^{n-1} + 2^{n-2} + J_{n-2} ) 
\right) 
 - (1 + 2^{n+1}) + 2^k \\
& \subseteq 
\left( 
\bbZ_{\leq 2^{n-2}} \cup \bbZ_{\geq 2^{n-1} + 2^{n-2} + 1} 
\right) 
 - (1 + 2^{n+1}) + 2^k \\
& = 
\bbZ_{\leq 2^{n-2}  - (1 + 2^{n+1}) + 2^k} \cup \bbZ_{\geq 2^{n-1} + 2^{n-2} + 1 - (1 + 2^{n+1}) + 2^k}  \\
& \subseteq 
\bbZ_{\leq 2^{n-2}  - (1 + 2^{n+1}) + 2^{n-2}} \cup \bbZ_{\geq 2^{n-1} + 2^{n-2} + 1 - (1 + 2^{n+1}) }  \\
& \subseteq 
\bbZ_{\leq 2^{n-1}  - (1 + 2^{n+1}) } \cup \bbZ_{\geq 2^{n-1} + 2^{n-2} + 1 - (1 + 2^{n+1}) }  
\end{align*}
hold for $k\leq n-2$, and the inclusions 
\begin{align*}
I_n + 2^k 
& = 
K_n - (1 + 2^{n+1}) + 2^k \\
& \subseteq 
\left( 
\{2^{n-3} + 1, 2^{n-3} + 2, \cdots, 2^{n-2} \} \cup (2^{n-1} + 2^{n-2} + J_{n-2} ) 
\right) 
 - (1 + 2^{n+1}) + 2^k \\
& \subseteq 
\left( 
\bbZ_{\geq 2^{n-3}+1} \cup \bbZ_{\geq 2^{n-1} + 2^{n-2} + 1} 
\right) 
 - (1 + 2^{n+1}) + 2^k \\
& \subseteq 
\bbZ_{\geq 2^{n-3}+1}
 - (1 + 2^{n+1}) + 2^k \\
& = 
\bbZ_{\geq 2^{n-3}+1
 - (1 + 2^{n+1}) + 2^k} \\
& \subseteq 
\bbZ_{\geq 2^{n-3}+1
 - (1 + 2^{n+1}) + 2^{n-1}} \\
& = 
\bbZ_{\geq 2^{n-1} +  2^{n-3}+1
 - (1 + 2^{n+1}) } 
\end{align*}
hold for $k\geq n-1$.

We claim that for any $n\geq 3$, no point in the left quarter of the right half of $\calI_n$ lie in $I_{n+1} + (\calT\setminus \{2^n\})$, i.e., 
\begin{equation}
\label{Eqn:1stQuarDeRightHalfNotObtained}
\left(
\{
2^{n-1} + 1, 2^{n-1} + 2, \cdots, 2^{n-1}+2^{n-3}
\} - (1 + 2^{n+1})
\right) 
\cap 
\left(
I_{n+1} + (\calT\setminus \{2^n\})
\right)
= \emptyset.
\end{equation}
This claim follows since for $n\geq 3$, the inclusions 
\begin{align*}
I_{n+1} + 2^k 
& \subseteq 
\calI_{n+1} + 2^k  \\
& = 
(\{1, 2, 3, \cdots, 2^{n+1} \} - ( 1+ 2^{n+2} ) ) + 2^k  \\
& \subseteq 
\bbZ_{\geq 1 - ( 1+ 2^{n+2} ) + 2^k} \\
& = 
\bbZ_{\geq 2^k -  2^{n+2} }\\
& \subseteq 
\bbZ_{\geq 0}
\end{align*}
hold for $k \geq n+2$, the inclusions 
\begin{align*}
I_{n+1} + 2^{n+1} 
& = 
K_{n+1} - (1 + 2^{n+2}) + 2^{n+1}  \\
& = 
\left( 
\{2^{n-2} + 1, 2^{n-2} + 2, \cdots, 2^{n-1} \} \cup (2^{n} + 2^{n-1} + J_{n-1} ) 
\right) 
 - (1 + 2^{n+2}) + 2^{n+1}  \\
& = 
\left( 
\{2^{n-2} + 1, 2^{n-2} + 2, \cdots, 2^{n-1} \} \cup (2^{n} + 2^{n-1} + J_{n-1} ) 
\right) 
 - (1 + 2^{n+1}) \\
& \subseteq 
\left(
\bbZ_{\leq 2^{n-1} } \cup \bbZ_{\geq 2^n + 2^{n-1} + 1} 
\right)
  - (1 + 2^{n+1}) \\
& \subseteq 
 \bbZ_{\leq 2^{n-1}   - (1 + 2^{n+1}) } \cup \bbZ_{\geq 2^n  +1 - (1 + 2^{n+1}) }
\end{align*}
hold, and the inclusions 
\begin{align*}
I_{n+1} + 2^k 
& \subseteq 
\calI_{n+1} + 2^k \\
& = 
(\{1, 2, 3, \cdots, 2^{n+1} \} - ( 1+ 2^{n+2} ) ) + 2^k  \\
& \subseteq 
\bbZ_{\leq 2^{n+1} - ( 1+ 2^{n+2} ) + 2^k} \\
& = 
\bbZ_{\leq 2^k - ( 1+ 2^{n+1} )} \\
& \subseteq 
\bbZ_{\leq 2^{n-1} - ( 1+ 2^{n+1} )} 
\end{align*}
hold for $k\leq n-1$.

We now show that Equation \eqref{Eqn:LeftQuarDeRightQuar} holds. From Equations \eqref{Eqn:RightHalfNotObtained}, \eqref{Eqn:1stQuarDeRightHalfNotObtainedIn}, it follows that 
$$(\scrS \cap I_{n+1}) + \scrT 
\supseteq 
\{
2^{n-1}+1, 2^{n-1}+2, 2^{n-1}+3,\cdots, 
2^{n-1} + 2^{n-3}
\} - (1 + 2^{n+1}).$$
Using Equation \eqref{Eqn:1stQuarDeRightHalfNotObtained}, it follows that $2^n \in \scrT$ and 
$$(\scrS \cap I_{n+1}) + 2^n
\supseteq 
\{
2^{n-1}+1, 2^{n-1}+2, 2^{n-1}+3,\cdots, 
2^{n-1} + 2^{n-3}
\} - (1 + 2^{n+1}).$$
So the inclusions 
\begin{align*}
\scrS
& \supseteq 
\{
2^{n-1}+1, 2^{n-1}+2, 2^{n-1}+3,\cdots, 
2^{n-1} + 2^{n-3}
\} - (1 + 2^{n+1}) - 2^n\\
& = 
\{
2^n + 2^{n-1}+1, 2^n + 2^{n-1}+2, 2^n + 2^{n-1}+3,\cdots, 
2^n + 2^{n-1} + 2^{n-3}
\} - (1 + 2^{n+1}) - 2^n - 2^n \\
& = 
\{
2^n + 2^{n-1}+1, 2^n + 2^{n-1}+2, 2^n + 2^{n-1}+3,\cdots, 
2^n + 2^{n-1} + 2^{n-3}
\} - (1 + 2^{n+2}).
\end{align*}
hold, which establishes Equation \eqref{Eqn:LeftQuarDeRightQuar}.

Note that $S + (\calT \setminus \{1\})$ does not contain $-1$. It follows that $a = 0$. So $1\in \scrT$ and $-2\in \scrS$. The integer $-5 - (1 + 2^5)$ does not lie in $S + (\calT \setminus\{2\})$. Hence $\scrT$ contains $2$. Note that $(S \setminus \{-4\}) + \calT$ does not contain $-3$. It follows that $-4\in \scrS$. Consequently, $\calT$ is a minimal complement of $S$. 
\end{proof}

By Theorem \ref{Thm:RequiredClusters}, it follows that $\calT$ is a minimal complement of $S$. However, it turns out that $S$ is not a minimal complement of $\calT$. In fact, 
$$S \setminus\{2^{2n} + 2^{2n-1} + 2^{2n-2} + \cdots + 2^4 + 2^3 + \frac{2^2}{2} - (1 + 2^{2n+2})\,|\, n\geq 3\}$$
is also an additive complement of $\calT$. In the following result, we prove that $S$ contains a subset $\calS$ such that $\calS$ is a minimal complement of $\calT$. 

\begin{theorem}
\label{Thm:CoMinS}
Let $S = \{ -x_1, - x_2, -x_3, \cdots\}$ where $x_1 < x_2 < x_3< \cdots$. Define $S_1 = S$ and for each positive integer $i\geq 1$, define 
$$S_{i+1} 
: = 
\begin{cases}
S_i \setminus \{-x_i\} & \text{ if $S_i \setminus \{-x_i\}$ is a complement to $\calT$,}\\
S_i & \text{ otherwise.}
\end{cases}
$$
Let $\calS$ denote the subset $\cap _{i \geq 1} S_i$ of $\bbZ$. The subsets $\calS$ and $\calT$ of $\bbZ$ form a co-minimal pair.
\end{theorem}

\begin{proof}
By Theorem \ref{Thm:RequiredClusters}, for $i\geq 1$, $S_i$ contains the points in the second quarter of the left half of $\calI_n$, i.e.,
\begin{equation}
S_i \supseteq \{2^{n-3} + 1, 2^{n-3} + 2, 2^{n-3} + 3, \cdots, 2^{n-2}\} - (1 + 2^{n+1})
\quad 
\text{ for }
n\geq 3,
\end{equation}
thus $\calS$ contains these points and hence $\bbZ_{\geq 1}$ is contained in $\calS + \calT$. 
By Theorem \ref{Thm:RequiredClusters}, for any $i\geq 1$, $S_i$ contains $-2, -4$. So $\{-3, -2, -1, 0\} \in \calS + \calT$. 
By Theorem \ref{Thm:RequiredClusters}, for any $i\geq 1$, $S_i$ contains the points in the left quarter of the right quarter of $\calI_{n}$, i.e.,  
\begin{equation}
S_i \supseteq 
\{2^{n-1}+ 2^{n-2} + 1, 2^{n-1}+ 2^{n-2} + 2, 2^{n-1}+ 2^{n-2} + 3, \cdots, 2^{n-1}+ 2^{n-2} + 2^{n-4}\}  - (1 + 2^{n+1})
\quad 
\text{ for }
n\geq 4,
\end{equation}
thus $\calS$ contains these points and hence the left half of $\calI_n$ is contained in $\calS + \calT$ for any $n\geq 1$. It remains to show that the second half of $\calI_n$ is also contained in $\calS + \calT$ for any $n\geq 2$. 

For any element $y\in \bbZ_{\leq -1}$ that lie in the right half of some $\calI_n$ for some $n\geq 2$ and for any $i\geq 1$, there exist elements $s_{y, i}\in S_i , t_{y, i}\in \calT$ such that $y = s_{y, i} + t_{y, i}$. 
By Equation \eqref{Eqn:RightHalfNotObtained}, it follows that for some element $s_y \in S$, the equality $s_y = s_{y, i}$ holds for infinitely many $i$ and for such integers $i$, we have $t_y = t_{y, i}$ where $t_y: = y - s_y\in \calT$. Thus $y - t_y = s_{y, i}$ holds for infinitely many $i$. Hence, for each integer $i\geq 1$, there exists an integer $m_i \geq i$ such that $y - t_y = s_{y, m_i}$, which yields $y\in t_y  + S_{m_i} \subseteq t_y + S_i$. Thus $y$ lies in $t_y +  \calS$. This proves that the second half of $\calI_n$ is also contained in $\calS + \calT$ for any $n\geq 2$. 
Hence $\calS$ is an additive complement of $\calT$. It follows that $\calS$ is a minimal complement to $\calT$. 

By Theorem \ref{Thm:RequiredClusters}, it follows that $\calT$ is a minimal complement to $\calS$. This proves that $(\calS, \calT)$ is a co-minimal pair. 
\end{proof}

\section{A co-minimal pair involving an infinite symmetric subset}

In this section, we establish Theorem \ref{Thm:CoMinUV}, which follows from Theorem \ref{Thm:CoMinU}. 
Consider the subsets $\{\calI_n\}_{n\geq 0}, \{U_n\}_{n\geq 0}$ of $\bbZ$ defined by 
\begin{align*}
\calI_n 
& = \begin{cases}
\{-2, -1\} & \text{ if } n = 0, \\
\{1, 2, 3, \cdots, 2^n\} - (1+2^{n+1}) & \text{ for } n \geq 1,
\end{cases} 
\\
U_n 
& = 
\begin{cases}
\{-2, -1\} & \text{ if } n = 0,\\
\emptyset & \text{ if } n = 1, 2, \\
\{6\} - ( 1+ 2^4) & \text{ if } n = 3,\\
\{3, 4, 14\} - ( 1+ 2^5) & \text{ if } n = 4,\\
((2^{n-3} + \{1,2, 3,  \cdots, 2^{n-3}\})
\cup 
(2^{n-1}+2^{n-2}+\{1,2, 3, \cdots, 2^{n-4}\}))-(1 + 2^{n+1}) & \text{ if } n\geq 5.
\end{cases} 
\end{align*}
Denote the union $\cup _{n\geq 0} U_n$ by $U$. 

\begin{proposition}
\label{Prop:VAddComp}
The set $\calV$ is an additive complement of $U$ in $\bbZ$. 
\end{proposition}

\begin{proof}
Since the inclusion
$$U_0 + \{2\} 
\supseteq  
\{0, 1\}$$
holds
and the inclusions 
\begin{align*}
\cup_{k = 4}^m (U_k + 2^{k+1})
& \supseteq 
\cup_{k = 4}^m (2^{k-3} + \{0, 1,2,  \cdots, 2^{k-3}-1\})\\
& \supseteq 
\{2, 3, 4, \cdots, 2^{m-2}-1\}
\end{align*}
hold for any $m\geq 4$, it follows that $U + \calV$ contains $\bbZ_{\geq 0}$. 

Note that 
$$U_n \subseteq \calI_n\quad \text{ for all }n\geq 0,$$
and 
$$\bigcup_{n\geq 0} \calI_n = \bbZ_{\leq -1},$$ 
and the sets $\calI_n$ lie next to each other in the sense that 
$$\min \calI_n = 1 + \max \calI_{n+1} \quad \text{ for all }n\geq 0.$$
Thus, to prove that $U+\calV$ is equal to $\bbZ$, it remains to show that $U + \calV$ contains $\calI_n$ for all $n\geq 0$. 

The left half of $\calI_n$ is contained in $U_{n+3} + 2^{n+3}$ for any $n\geq 1$, i.e., the inclusions
\begin{align*}
U_{n+3} + 2^{n+3} 
& \supseteq (2^{n+2} + 2^{n+1} + \{1, 2, 3, \cdots, 2^{n-1}\}) - ( 1 + 2^{n+4}) + 2^{n+3}\\
& = (\{1, 2, 3, \cdots, 2^{n-1}\} - (1 + 2^{n+1})) + (1 + 2^{n+1}) + 2^{n+2} + 2^{n+1}- ( 1 + 2^{n+4}) + 2^{n+3}\\
& = \{1, 2, 3, \cdots, 2^{n-1}\}  - (1 + 2^{n+1})
\end{align*}
hold for any $n\geq 1$. The third quarter of $\calI_n$ is contained in the set $(U_{n+1} + 2^n )\cup (U_n + 2^{n-1})$ for $n\geq 4$, i.e., the inclusions
\begin{align*}
 (U_{n+1} + 2^n )\cup (U_n + 2^{n-1})
& \supseteq 
\left((2^n + 2^{n-1} + \{1, 2, 3, \cdots, 2^{n-3}\} ) - (1 + 2^{n+2})+ 2^n
\right) \\
& \qquad 
\cup 
\left(
(2^{n-3} + \{1, 2, 3, \cdots, 2^{n-3}\}) - (1 + 2^{n+1})
+ 2^{n-1}
\right)\\
& \supseteq 
((2^{n-1} + \{1, 2, 3, \cdots, 2^{n-3}\} ) - (1 + 2^{n+1})) \\
& \qquad 
\cup 
(
(2^{n-1}+2^{n-3} + \{1, 2, 3, \cdots, 2^{n-3}\}) - (1 + 2^{n+1}))\\
& \supseteq 
(2^{n-1} + \{1, 2, 3, \cdots, 2^{n-2}\} ) - (1 + 2^{n+1})
\end{align*}
hold for $n\geq 4$. The left half of the right quarter of $\calI_4$ is contained in the set 
$(U_3 + (-2^3))\cup (U_4 + (-1))$, 
and the left half of the right quarter of $\calI_n$ is contained in the set 
$(U_n + (-2^{n-5}))
\cup 
(U_n + 2^{n-5})
\cup 
(U_n + 2^{n-4})$ for $n\geq 5$, i.e., the inclusions
\begin{align*}
& (U_n + (-2^{n-5}))
\cup 
(U_n + 2^{n-5})
\cup 
(U_n + 2^{n-4})\\
& \supseteq 
((2^{n-1} + 2^{n-2} + \{1, 2, 3, \cdots, 2^{n-4}\}) - (1 + 2^{n+1})) 
+ \{-2^{n-5}, 2^{n-5}, 2^{n-4} \} \\
& \supseteq 
\left(
((2^{n-1} + 2^{n-2} + \{2^{n-5}+1,2^{n-5}+ 2,2^{n-5}+ 3, \cdots, 2^{n-4}\}) - (1 + 2^{n+1})) 
-2^{n-5} 
\right)\\
& \qquad \cup
\left(
((2^{n-1} + 2^{n-2} + \{1, 2, 3, \cdots, 2^{n-5}\}) - (1 + 2^{n+1})) 
+ 2^{n-5}
\right)\\
& \qquad \cup
\left(
((2^{n-1} + 2^{n-2} + \{1, 2, 3, \cdots, 2^{n-4}\}) - (1 + 2^{n+1})) 
+ 2^{n-4} 
\right)\\
& \supseteq 
((2^{n-1} + 2^{n-2} + \{1, 2, 3, \cdots, 2^{n-4}\}) - (1 + 2^{n+1})) \\
& \qquad 
\cup 
((2^{n-1} + 2^{n-2} + 2^{n-4} + \{1, 2, 3, \cdots, 2^{n-4}\}) - (1 + 2^{n+1})) \\
& \supseteq 
(2^{n-1} + 2^{n-2} + \{1, 2, 3, \cdots, 2^{n-3}\}) - (1 + 2^{n+1}),
\end{align*}
hold for $n\geq 5$. The right half of the right quarter of $\calI_4$ is contained in the set $U_0 + (-2^4)$, and the right half of the right quarter of $\calI_n$ is contained in the set $(U_n + 2^{n-3}) \cup (U_{n-1} - 2^{n-3} )$ for $n\geq 5$, i.e., the inclusions
\begin{align*}
& (U_n + 2^{n-3}) \cup (U_{n-1} - 2^{n-3} ) \\
& \supseteq 
\left(
((2^{n-1} + 2^{n-2} + \{1, 2, 3, \cdots, 2^{n-4}\}) - (1 + 2^{n+1})) 
+ 2^{n-3}
\right) \\
& \qquad 
\cup 
\left( 
((2^{n-4} + \{1, 2, 3, \cdots, 2^{n-4}\}) - (1 + 2^n) ) - 2^{n-3} 
\right) \\
& = 
\left((
(2^{n-1} + 2^{n-2} + 2^{n-3} + \{1, 2, 3, \cdots, 2^{n-4}\}) - (1 + 2^{n+1})
\right) \\
& \qquad 
\cup 
\left(
((2^{n-1} + 2^{n-2} + 2^{n-3} + 2^{n-4} + \{1, 2, 3, \cdots, 2^{n-4}\}) - (1 + 2^n) ) - 2^{n-3} - 2^{n-1} - 2^{n-2} - 2^{n-3} 
\right)\\
& = 
\left((
(2^{n-1} + 2^{n-2} + 2^{n-3} + \{1, 2, 3, \cdots, 2^{n-4}\}) - (1 + 2^{n+1})
\right) \\
& \qquad 
\cup 
\left(
(2^{n-1} + 2^{n-2} + 2^{n-3} + 2^{n-4} + \{1, 2, 3, \cdots, 2^{n-4}\}) - (1 + 2^{n+1})
\right)\\
& = 
(2^{n-1} + 2^{n-2} + 2^{n-3} + \{1, 2, 3, \cdots, 2^{n-3}\}) - (1 + 2^{n+1})
\end{align*}
hold for $n\geq 5$. This proves that $\calI_n$ is contained in $U + \calV$ for $n\geq 4$, and hence $U + \calV$ contains $\bbZ_{\leq 2^4 - (1 + 2^5)} = \bbZ_{\leq - 1 - 2^4}$. Also note that the left half of $\calI_1, \calI_2, \calI_3$ is contained in $U + \calV$. From the inclusions
\begin{align*}
& (U_4 \cup U_3 \cup U_0) 
+ 
\{-4, -2, -1, 1, 2, 4, 8\}\\
& \supseteq 
\left(
(U_4 + \{8\}) 
\cup 
(U_3 + \{-1, 1, 2\} )
\right)
\cup 
(U_0 + \{-4\})
\cup 
(U_0 + \{-2\})
\cup 
(U_0 + \{-1, 1\})
\\
& \supseteq 
(\{5, 6, 7, 8\} - ( 1 + 2^4))
\cup 
(\{3, 4\} - (1+ 2^3))
\cup 
(\{1, 2\} - ( 1+ 2^2))
\cup 
\{-2, -1\}, 
\end{align*}
it follows that the subsets $\calI_0, \calI_1$ and the right half of $\calI_2, \calI_3$ is contained in $U + \calV$. Hence $\calV$ is an additive complement of $U$. 
\end{proof}

\begin{theorem}
\label{Thm:RequiredClustersV}
The set $\calV$ is a minimal complement of $U$ in $\bbZ$. 
\end{theorem}

\begin{proof}
By Proposition \ref{Prop:VAddComp}, $\calV$ is an additive complement of $U$. To prove that $\calV$ is a minimal complement of $U$ in $\bbZ$, it suffices to prove that if $\scrV$ is nonempty subset of $\calV$ satisfying $U + \scrV = \bbZ$, then $\scrV = \calV$, which follows from the nine claims below. Indeed, the third claim below implies that $2^n\in \scrV$ for all $n\geq 2$, the fourth (resp. fifth, sixth, seventh, eighth) claim implies that $\scrV$ contains $2$ (resp. $1, -1, -2, -4$), and the ninth claim implies that $\scrV$ contains $-2^n$ for all $n\geq 3$. Thus, from the following claims, it follows that $\scrV = \calV$ and hence $\calV$ is a minimal complement to $U$. 

We claim the following.  

\begin{enumerate}[\textbf{\text{Claim }} 1\textbf{.}]
\item 
For any $n\geq 2$, no point in the right half of $\calI_n$ lie in $\cup _{m \geq n+2}U_m + \calV$, i.e., 
\begin{equation}
\label{Eqn:RightHalfNotObtainedUV}
\left(
\{2^{n-1}+1, 2^{n-1}+2, 2^{n-1}+3, \cdots, 2^n\} - (1 + 2^{n+1})
\right) 
\cap 
\left(
\cup _{m \geq n+2}U_m + \calV
\right)
= \emptyset
\quad
\text{ for } 
n\geq 2.
\end{equation}

\item 
For any $n\geq 3$, no point in the right half of $\calI_n$ lie in $U_{n+1} + (\calV \setminus\{2^n\})$, i.e., 
\begin{equation}
\label{Eqn:RightHalfNotObtainedLHS1ptUVbis}
\left(
\{2^{n-1}+1, 2^{n-1}+2, 2^{n-1}+3, \cdots, 2^n\} - (1 + 2^{n+1})
\right) 
\cap 
\left(
U_{n+1} + (\calV \setminus \{ 2^n\} )
\right)
= \emptyset
\quad
\text{ for } 
n\geq 3.
\end{equation}

\item 
For any $n\geq 5$, the largest element of the third quarter of the right quarter of $\calI_n$ does not lie in $U + (\calV \setminus \{2^{n-3} \})$, more precisely, 
\begin{equation}
\label{Eqn:VContainsPositifPowersOf2}
2^{n-1} + 2^{n-2} + 2^{n-3} + 2^{n-4} - (1 + 2^{n+1}) \notin (U_n + (\calV \setminus\{2^{n-3}\})) \cup (\cup_{m \geq 0, m \neq n} U_m + \calV)
\quad 
\text{ for } n\geq 5.
\end{equation}

\item 
\begin{equation}
\label{Eqn:VContains2}
1 \notin ((U \setminus \{-1\}) + \calV) \cup (-1 + (\calV \setminus \{2\})) .
\end{equation}

\item 
\begin{equation}
\label{Eqn:VContains1}
26 - (1 + 2^6) \notin ((U \setminus \{25 - (1+2^6)\}) + \calV) \cup (25 - (1 + 2^6) + (\calV \setminus \{1\})).
\end{equation}

\item 
\begin{equation}
\label{Eqn:VContainsMinus1}
25 - (1 + 2^6) \notin ((U \setminus \{26 -(1+2^6)\}) + \calV) \cup (26 - (1 + 2^6) + (\calV \setminus \{-1\})).
\end{equation}

\item 
\begin{equation}
\label{Eqn:VContainsMinus2}
-4 \notin ((U \setminus \{-2\}) + \calV) \cup (-2 + (\calV \setminus \{-2\}))  .
\end{equation}

\item 
\begin{equation}
\label{Eqn:VContainsMinus4}
-6 \notin ((U \setminus \{-2\}) + \calV) \cup (-2 + (\calV \setminus \{-4\}))  .
\end{equation}

\item 
For any $n\geq 6$, the third largest element of $\calI_n$, i.e., the integer $2^n - 2 - (1 + 2^{n+1})$ does not belong to $U + (\calV\setminus \{-2^{n-3}\})$, i.e., 
\begin{equation}
\label{Eqn:VContainsNegativePowersOf2}
2^n - 1 - (1 + 2^{n+1}) \notin U + (\calV\setminus \{-2^{n-3}\}).
\end{equation}
\end{enumerate}

First, we establish Equation \eqref{Eqn:RightHalfNotObtainedUV}. 
For any $n\geq 2$, the inclusion $U_n \subseteq \calI_n = \{1, 2, 3, \cdots, 2^n\}- (1 + 2^{n+1})$ yields that 
$$U_n \cap \bbZ_{\geq 0} = \emptyset, U_n \cap \bbZ_{\leq - 2^{n+1} - 1} = \emptyset,$$
and for any $m\geq n+2$ with $n\geq 2$, the inclusions
\begin{align*}
U_m + v
& \subseteq 
\calI_m + v \\
& \subseteq 
(\{1, 2, 3, \cdots, 2^m\} - (1 + 2^{m+1})) + v \\
& = 
\{1, 2, 3, \cdots, 2^m\}  + v-  2^{m+1} - 1\\
& \subseteq \bbZ_{\geq 0}
\end{align*}
hold for any $v\in \calV$ with $v \geq 2^{m+1}$, and the inclusions
\begin{align*}
U_m + v
& \subseteq 
\bbZ_{\leq 2^m - (1 + 2^{m+1})} + v \\
& = \bbZ_{\leq 2^m - (1 + 2^{m+1}) + v } \\
& = \bbZ_{\leq -1-2^m + v} \\
& \subseteq 
\bbZ_{\leq  - 1 - 2^m + 2^{m-1}} \\
& \subseteq 
\bbZ_{\leq  - 1 - 2^{m-1}} \\
&\subseteq
\bbZ_{\leq - (1+2^{n+1})}
\end{align*}
hold for any $v\in \calV$ with $v \leq 2^{m-1}$, and hence 
\begin{equation}
\label{Eqn:RightHalfNotObtainedLHSSans1ptUV}
\left(
\{2^{n-1}+1, 2^{n-1}+2, 2^{n-1}+3, \cdots, 2^n\} - (1 + 2^{n+1})
\right) 
\cap 
\left(
\cup_{m \geq n+2} (U_m + (\calV \setminus \{2^m\}))
\right)
= \emptyset
\quad
\text{ for } 
n\geq 2.
\end{equation}
For any $m\geq 4$, the inclusions
\begin{align*}
U_m + 2^m 
& \subseteq 
((2^{m-3} + \{1,2, 3,  \cdots, 2^{m-3}\})
\cup 
(2^{m-1}+2^{m-2}+\{1,2, 3, \cdots, 2^{m-4}\}))-(1 + 2^{m+1}) + 2^m\\
& = 
((2^{m-3} + \{1,2, 3,  \cdots, 2^{m-3}\})
\cup 
(2^{m-1}+2^{m-2}+\{1,2, 3, \cdots, 2^{m-4}\}))-(1 + 2^m)\\
& \subseteq
(\{1,2, 3,  \cdots, 2^{m-2}\} - (1 + 2^m))
\cup 
(2^{m-1}+2^{m-2}+\{1,2, 3, \cdots, 2^{m-4}\}-(1 + 2^m))\\
& =
(\{1,2, 3,  \cdots, 2^{m-2}\} - (1 + 2^m))
\cup 
(\{1,2, 3, \cdots, 2^{m-4}\}-(1 + 2^{m-2}))
\end{align*}
hold, which implies that for $n\geq 3$ and $m\geq n+1$, the set $U_m + 2^m$ is contained in the union of the left half of $\calI_{m-1}$ and the left half of $\calI_{m-3}$, and hence 
\begin{equation}
\label{Eqn:RightHalfNotObtainedLHS1ptUV}
\left(
\{2^{n-1}+1, 2^{n-1}+2, 2^{n-1}+3, \cdots, 2^n\} - (1 + 2^{n+1})
\right) 
\cap 
\left(
\cup_{m \geq n+1} (U_m + 2^m)
\right)
= \emptyset
\quad
\text{ for } 
n\geq 3.
\end{equation}
From Equations \eqref{Eqn:RightHalfNotObtainedLHSSans1ptUV}, \eqref{Eqn:RightHalfNotObtainedLHS1ptUV}, Equation \eqref{Eqn:RightHalfNotObtainedUV} follows. 

Now, we prove that Equation \eqref{Eqn:RightHalfNotObtainedLHS1ptUVbis} holds. 
For any $m\geq 4$, the inclusions
\begin{align*}
U_m + v 
& \subseteq 
\bbZ_{\geq 1 - (1 + 2^{m+1})} + v \\
& = 
\bbZ_{\geq 1 - (1 + 2^{m+1}) + v }\\
& \subseteq 
\bbZ_{\geq 1 - (1 + 2^{m+1}) + 2^{m+1} }\\
& = \bbZ_{\geq 0}
\end{align*}
hold for $v\in \calV$ with $v\geq 2^{m+1}$, and the inclusions
\begin{align*}
U_m + v 
& \subseteq 
\bbZ_{\leq 2^m - (1 + 2^{m+1})} + v \\
& = 
\bbZ_{\leq 2^m - (1 + 2^{m+1}) + v }\\
& \subseteq 
\bbZ_{\leq 2^m - (1 + 2^{m+1}) + 2^{m-2} }\\
& = \bbZ_{\leq 2^{m-2} -(1 + 2^{m})}
\end{align*}
hold for $v\in \calV$ with $v \leq  2^{m-2}$, and thus Equation \eqref{Eqn:RightHalfNotObtainedLHS1ptUV} yields Equation \eqref{Eqn:RightHalfNotObtainedLHS1ptUVbis}. 

Now, we show that Equation \eqref{Eqn:VContainsPositifPowersOf2} holds. Note that the inclusions
\begin{align*}
\calI_n \cap (\cup_{0 \leq m < n} U_m + v)
& \subseteq 
\calI_n \cap (\cup_{0 \leq m < n} \calI_m + v)\\
& \subseteq 
\calI_n \cap (\bbZ_{\leq -1} +v )\\
& \subseteq 
\bbZ_{\geq 1 - (1 + 2^{n+1})}\cap (\bbZ_{\leq -(1 + 2^{n+1})}) \\
& \subseteq \emptyset
\end{align*}
hold for any $n\geq 5$ and for any $v\in \calV$ with $v\leq - 2^{n+1}$. 
The inclusions
\begin{align*}
& (\{2^{n-1} + 2^{n-2} + 2^{n-3}+k\,|\,1 \leq k \leq 2^{n-4}\} - (1 + 2^{n+1}) ) 
\cap (\cup_{0 \leq m < n}U_m  - 2^n )\\
& \subseteq  
(\{2^{n-1} + 2^{n-2} + 2^{n-3}+k\,|\,1 \leq k \leq 2^{n-4}\} - (1 + 2^{n+1}) ) \\
& \qquad 
\cap (((U_{n-1} \cup U_{n-2} \cup U_{n-3} )\cup U_{n-4}\cup (\cup_{0 \leq m \leq n-5}U_m))  - 2^n )\\
& \subseteq  
(\{2^{n-1} + 2^{n-2} + 2^{n-3}+k\,|\,1 \leq k \leq 2^{n-4}\} - (1 + 2^{n+1}) ) \\
& \qquad 
\cap (((\calI_{n-1} \cup \calI_{n-2} \cup \calI_{n-3} )\cup U_{n-4} \cup(\cup_{0 \leq m \leq n-5}\calI_m))  - 2^n )\\
& \subseteq  
(\{2^{n-1} + 2^{n-2} + 2^{n-3}+k\,|\,1 \leq k \leq 2^{n-4}\} - (1 + 2^{n+1}) ) \\
& \qquad 
\cap (((\bbZ_{\leq 2^{n-4}-1 - (1 + 2^{n-3})})\cup  (\bbZ_{\geq 1 - ( 1+ 2^{n-4}) }))  - 2^n )\\
& \subseteq  
(\{2^{n-1} + 2^{n-2} + 2^{n-3}+k\,|\,1 \leq k \leq 2^{n-4}\} - (1 + 2^{n+1}) ) \\
& \qquad 
\cap ( (\bbZ_{\leq 2^{n-4}-1 - (1 + 2^{n-3})} - 2^n)\cup  ((\bbZ_{\geq 1 - ( 1+ 2^{n-4}) }) - 2^n))\\
& \subseteq  
(\{2^{n-1} + 2^{n-2} + 2^{n-3}+k\,|\,1 \leq k \leq 2^{n-4}\} - (1 + 2^{n+1}) ) \\
& \qquad 
\cap ( \bbZ_{\leq 2^{n-4} -1- (1 + 2^{n-3})- 2^n}\cup \bbZ_{\geq 1 - ( 1+ 2^{n-4})- 2^n })\\
& \subseteq  
(\{2^{n-1} + 2^{n-2} + 2^{n-3}+k\,|\,1 \leq k \leq 2^{n-4}\} - (1 + 2^{n+1}) ) \\
& \qquad 
\cap ( \bbZ_{\leq 2^{n-1} + 2^{n-2} + 2^{n-3} + 2^{n-4} - 1 - (1 + 2^{n+1})}\cup \bbZ_{\geq 2^{n-1} + 2^{n-2} + 2^{n-3} + 2^{n-4} +  1 - ( 1+ 2^{n+1})})\\
& \subseteq  
(\{2^{n-1} + 2^{n-2} + 2^{n-3}+k\,|\,1 \leq k \leq 2^{n-4}\} - (1 + 2^{n+1}) ) \\
& \qquad 
\cap  \bbZ_{\leq 2^{n-1} + 2^{n-2} + 2^{n-3} + 2^{n-4} - 1 - (1 + 2^{n+1})}\\
& \subseteq 
\bbZ_{\leq 2^{n-1} + 2^{n-2} + 2^{n-3} + 2^{n-4} - 1 - (1 + 2^{n+1})}
\end{align*}
hold for $n\geq 5$. 
The inclusions
\begin{align*}
& \calI_n \cap (\cup_{0 \leq m < n}U_m  - 2^{n-1} )\\
& \subseteq  
(U_{n-1} - 2^{n-1} )  
\cup 
(\calI_n \cap (\cup_{0 \leq m \leq  n-2}U_m  - 2^{n-1} ))\\
& \subseteq 
\left(
 ( ((2^{n-4} + \{1, 2, 3, \cdots , 2^{n-4} \}) \cup (2^{n-2} + 2^{n-3} + \{1, 2, 3, \cdots, 2^{n-5}\}))    - (1 + 2^n) )- 2^{n-1} 
 \right) \\
& \qquad  
\cup 
(\calI_n \cap (\cup_{0 \leq m \leq  n-2}\calI_m  - 2^{n-1} ))\\
& \subseteq 
\bigl(
 ((2^{n-4} + \{1, 2, 3, \cdots , 2^{n-4} \}) - ( 1+ 2^n + 2^{n-1})) \\
& \qquad \cup ((2^{n-2} + 2^{n-3} + \{1, 2, 3, \cdots, 2^{n-5}\}) - ( 1+ 2^n + 2^{n-1})))    
 \bigr) 
\cup 
(\calI_n \cap (\bbZ_{\geq 1 - (1 + 2^{n-1})} - 2^{n-1}))\\
& = 
\bigl(
 ((2^{n-1}  + 2^{n-4} + \{1, 2, 3, \cdots , 2^{n-4} \}) - ( 1+ 2^{n+1})) \\
& \qquad 
\cup ((2^{n-1} + 2^{n-2} + 2^{n-3} + \{1, 2, 3, \cdots, 2^{n-5}\}) - ( 1+ 2^{n+1})))    
\bigr)
\cup 
(\calI_n \cap (\bbZ_{\geq -2^n}))\\
& = 
\left(
\bbZ_{\leq 2^{n-1}  + 2^{n-4} + 2^{n-4 } - ( 1+ 2^{n+1})}
\cup 
\bbZ_{\leq 2^{n-1} + 2^{n-2} + 2^{n-3}+ 2^{n-5} - ( 1+ 2^{n+1})} 
\right)
\cup 
(\calI_n \cap (\bbZ_{\geq 2^n + 1 - ( 1 + 2^{n+1})}))\\
& = 
\bbZ_{\leq 2^{n-1}  + 2^{n-3} - ( 1+ 2^{n+1})}
\cup 
\bbZ_{\leq 2^{n-1} + 2^{n-2} + 2^{n-3}+ 2^{n-5} - ( 1+ 2^{n+1})}
\end{align*}
hold for any $n\geq 5$. 
The inclusions
\begin{align*}
& \calI_n \cap (\cup_{0 \leq m < n}U_m  - 2^{n-2} )\\
& \subseteq  
(U_{n-1} - 2^{n-2} )  
\cup 
(\calI_n \cap (\cup_{0 \leq m \leq  n-2}U_m  - 2^{n-2} ))\\
& \subseteq 
\left(
 ( ((2^{n-4} + \{1, 2, 3, \cdots , 2^{n-4} \}) \cup (2^{n-2} + 2^{n-3} + \{1, 2, 3, \cdots, 2^{n-5} + 1\}))    - (1 + 2^n) )- 2^{n-2} 
 \right) \\
& \qquad  
\cup 
(\calI_n \cap (\cup_{0 \leq m \leq  n-2}\calI_m  - 2^{n-2} ))\\
& \subseteq 
\bigl(
 ((2^{n-4} + \{1, 2, 3, \cdots , 2^{n-4} \}) - ( 1+ 2^n + 2^{n-2})) \\
& \qquad \cup ((2^{n-2} + 2^{n-3} + \{1, 2, 3, \cdots, 2^{n-5}+1\}) - ( 1+ 2^n + 2^{n-2})))    
 \bigr) 
\cup 
(\calI_n \cap (\bbZ_{\geq 1 - (1 + 2^{n-1})} - 2^{n-2}))\\
& = 
\bigl(
 ((2^{n-1} + 2^{n-2} + 2^{n-4} + \{1, 2, 3, \cdots , 2^{n-4} \}) - ( 1+ 2^{n+1})) \\
& \qquad 
\cup ((2^{n-3} + \{1, 2, 3, \cdots, 2^{n-5}+1\}) - ( 1+ 2^n)))    
\bigr)
\cup 
(\calI_n \cap (\bbZ_{\geq -2^{n-1} - 2^{n-2}}))\\
& = 
\bigl(
 ((2^{n-1} + 2^{n-2} + 2^{n-4} + \{1, 2, 3, \cdots , 2^{n-4} \}) - ( 1+ 2^{n+1})) \\
& \qquad 
\cup ((2^{n-1} + 2^{n-2} + 2^{n-3} + 2^{n-4} + \{1, 2, 3, \cdots, 2^{n-5}+1\}) - ( 1+ 2^{n+1})))    
\bigr) \\
& \qquad 
\cup 
(\calI_n \cap (\bbZ_{\geq 2^n +  (1 + 2^{n-2})    - (1 + 2^{n+1}) }))\\
& = 
\bigl(
 ((2^{n-1} + 2^{n-2} + 2^{n-4} + \{1, 2, 3, \cdots , 2^{n-4} \}) - ( 1+ 2^{n+1})) \\
& \qquad 
\cup ((2^{n-1} + 2^{n-2} + 2^{n-3} + 2^{n-4} + \{1, 2, 3, \cdots, 2^{n-5}+1\}) - ( 1+ 2^{n+1})))    
\bigr)
\end{align*}
hold for any $n\geq 5$. 
The inclusions 
\begin{align*}
\cup_{0 \leq m < n}U_m + v
& \subseteq 
\bbZ_{\geq 2^{n-4}+1 - (1+2^{n}) } + v \\
& \subseteq 
\bbZ_{\geq 2^{n-4}+1 - (1+2^{n}) - 2^{n-3}}\\
& \subseteq 
\bbZ_{\geq 2^{n-1} + 2^{n-2} + 2^{n-3} + 2^{n-4} + 1- ( 1+ 2^{n+1})}
\end{align*}
hold for any $n\geq 5$ and for any $v\in \calV$ with $v\geq - 2^{n-3}$.
These inclusions prove that for $n\geq 5$, the largest element of the third quarter of the right quarter of $\calI_n$ does not belong to $\cup_{0 \leq m < n} U_m + \calV$, i.e., 
$$2^{n-1} + 2^{n-2} + 2^{n-3} + 2^{n-4} - (1 + 2^{n+1}) \notin \cup_{0 \leq m < n} U_m + \calV.$$
Combining the above with Equations \eqref{Eqn:RightHalfNotObtainedUV}, \eqref{Eqn:RightHalfNotObtainedLHS1ptUVbis}, we obtain 
$$
2^{n-1} + 2^{n-2} + 2^{n-3} + 2^{n-4} - (1 + 2^{n+1}) \notin 
\left(\cup_{m \geq 0, m \neq n, n+1} U_m + \calV\right) 
\cup 
(U_{n+1} + (\calV\setminus\{2^n\}))
\quad 
\text{ for } n\geq 5.
$$
Since the inclusions
\begin{align*}
U_{n+1} + 2^n 
& \subseteq 
\bbZ_{\leq 2^n + 2^{n-1} + 2^{n-3} - (1 + 2^{n+2})} + 2^n\\
& \subseteq 
\bbZ_{\leq 2^n + 2^{n-1} + 2^{n-3} - (1 + 2^{n+2}) + 2^n}\\
& \subseteq 
\bbZ_{\leq 2^{n-1} + 2^{n-3} - (1 + 2^{n+1}) }\\
\end{align*}
hold for $n\geq 5$, it follows that 
$$
2^{n-1} + 2^{n-2} + 2^{n-3} + 2^{n-4} - (1 + 2^{n+1}) \notin 
\left(\cup_{m \geq 0, m \neq n} U_m + \calV\right) 
\quad 
\text{ for } n\geq 5.
$$
Note that the inclusions 
\begin{align*}
U_n + v
& \subseteq 
\bbZ_{\geq 1 - (1 + 2^{n+1})} + v \\
& \subseteq 
\bbZ_{\geq 1 - (1 + 2^{n+1}) + v} \\
& \subseteq 
\bbZ_{\geq 1 - (1 + 2^{n+1}) + 2^n} \\
& \subseteq 
\bbZ_{\geq 1 - (1 + 2^{n}) } \\
\end{align*}
hold for any $v\in \calV$ with $v \geq 2^n$, 
the inclusions 
\begin{align*}
U_n + 2^{n-1}
& \subseteq 
(\bbZ_{\leq 2^{n-3} + 2^{n-3} - (1 + 2^{n+1})} \cup \bbZ_{\geq 2^{n-1} + 2^{n-2} + 1 - (1 + 2^{n+1})}) + 2^{n-1} \\
& \subseteq 
(\bbZ_{\leq 2^{n-2} - (1 + 2^{n+1})} \cup \bbZ_{\geq 2^{n-1} + 2^{n-2} + 1 - (1 + 2^{n+1})}) + 2^{n-1} \\
& \subseteq 
\bbZ_{\leq 2^{n-2} - (1 + 2^{n+1}) + 2^{n-1}} \cup \bbZ_{\geq 2^{n-1} + 2^{n-2} + 1 - (1 + 2^{n+1})+ 2^{n-1}} \\
& \subseteq 
\bbZ_{\leq 2^{n-1}+2^{n-2} - (1 + 2^{n+1}) } \cup \bbZ_{\geq 1 + 2^{n-2} - (1 + 2^{n})}
\end{align*}
hold, the inclusions
\begin{align*}
U_n + 2^{n-2}
& \subseteq 
(\bbZ_{\leq 2^{n-3} + 2^{n-3} - (1 + 2^{n+1})} \cup \bbZ_{\geq 2^{n-1} + 2^{n-2} + 1 - (1 + 2^{n+1})}) + 2^{n-2} \\
& \subseteq 
(\bbZ_{\leq 2^{n-2} - (1 + 2^{n+1})} \cup \bbZ_{\geq 2^{n-1} + 2^{n-2} + 1 - (1 + 2^{n+1})}) + 2^{n-2} \\
& \subseteq 
\bbZ_{\leq 2^{n-2} - (1 + 2^{n+1}) + 2^{n-2}} \cup \bbZ_{\geq 2^{n-1} + 2^{n-2} + 1 - (1 + 2^{n+1})+ 2^{n-2}} \\
& \subseteq 
\bbZ_{\leq 2^{n-3} - (1 + 2^{n+1}) } \cup \bbZ_{\geq 1 - (1 + 2^{n})}
\end{align*}
hold, and the inclusions 
\begin{align*}
U_n + v
& \subseteq 
\bbZ_{2^{n-1} + 2^{n-2} + 2^{n-4} - (1 + 2^{n+1})} + v\\
& \subseteq 
\bbZ_{2^{n-1} + 2^{n-2} + 2^{n-4} - (1 + 2^{n+1}) + v}\\
& \subseteq 
\bbZ_{2^{n-1} + 2^{n-2} + 2^{n-4} - (1 + 2^{n+1}) + 2^{n-4}}\\
& \subseteq 
\bbZ_{2^{n-1} + 2^{n-2} + 2^{n-3} - (1 + 2^{n+1})}\\
\end{align*}
hold for any $v\in \calV$ with $v\leq 2^{n-4}$. This yields Equation \eqref{Eqn:VContainsPositifPowersOf2}. 

Now, we show that Equation \eqref{Eqn:VContains2} holds. 
The inclusions
\begin{align*}
U_n + v
& \subseteq 
\bbZ_{\leq 2^n - ( 1+ 2^{n+1})} + v \\
& \subseteq 
\bbZ_{\leq 2^n - ( 1+ 2^{n+1})+ v}  \\
& \subseteq 
\bbZ_{\leq 2^n - ( 1+ 2^{n+1})+ 2^n}  \\
& \subseteq 
\bbZ_{\leq -1}
\end{align*}
hold for any $n\geq 1$ and $v\in \calV$ with $v \leq 2^n$. 
The inclusions 
\begin{align*}
U_n + 2^{n+1} 
& \subseteq 
\bbZ_{\geq 3 - (1 + 2^{n+1})} + 2^{n+1} \\
& \subseteq 
\bbZ_{\geq 3 - (1 + 2^{n+1}) + 2^{n+1}} \\
& \subseteq 
\bbZ_{\geq 2} \\
\end{align*}
hold for any $n\geq 1$. 
The inclusions
\begin{align*}
U_n + v
& \subseteq 
\bbZ_{\geq 1 - ( 1+ 2^{n+1})} + v \\
& \subseteq 
\bbZ_{\geq 1 - ( 1+ 2^{n+1})+ v}  \\
& \subseteq 
\bbZ_{\geq 1 - ( 1+ 2^{n+1})+ 2^{n+2}}  \\
& \subseteq 
\bbZ_{\geq 2^{n+1}}
\end{align*}
hold for any $n\geq 1$ and $v\in \calV$ with $v \geq 2^{n+2}$. 
This proves that 
$$1 \notin (U\setminus U_0) + \calV = (U \setminus \{-2, -1\}) + \calV.$$
Note that
\begin{align*}
-2 + \calV
& \subseteq 
(-2 + \{1\}) \cup (-2 + (\calV \setminus \{1\})) \\
& \subseteq 
\{-1\} \cup (-2 + (\calV \setminus \{1\})) 
\end{align*}
and the elements of $-2 + (\calV \setminus \{1\})$ are even. So $1$ does not belong to $-2 + \calV$. So $1 \notin (U \setminus \{-1\}) + \calV$, and hence Equation \eqref{Eqn:VContains2} holds. 

Now, we prove that Equations
\eqref{Eqn:VContains1}, 
\eqref{Eqn:VContainsMinus1} 
hold. 
By Equations \eqref{Eqn:RightHalfNotObtainedUV}, \eqref{Eqn:RightHalfNotObtainedLHS1ptUVbis}, we obtain 
$$
(\{25, 26\} - ( 1+ 2^6) )
\cap 
\left(
(U_6 + (\calV \setminus \{2^5\}))
\cup 
(\cup_{m\geq 7} U_m + \calV)
\right)
= 
\emptyset.
$$
It follows that 
$$
(\{25, 26\} - ( 1+ 2^6) )
\cap 
(U_6 + 2^5)
= \emptyset,$$
$$
(\{25, 26\} - ( 1+ 2^6) )
\cap 
((U_0 \cup \cdots \cup U_4) + \calV )
= \emptyset,$$
and hence 
$$
(\{25, 26\} - ( 1+ 2^6) )
\cap 
(\cup_{m\neq  5} U_m + \calV)
= 
\emptyset.
$$
It also follows that 
$$
(\{ 26\} - ( 1+ 2^6) )
\cap 
((U_5 \setminus \{25 - ( 1 + 2^6)\}) + \calV)
= \emptyset,$$
$$
(\{ 26\} - ( 1+ 2^6) )
\cap 
(25 - ( 1 + 2^6) + (\calV\setminus \{1\}))
= \emptyset,$$
$$
(\{ 25\} - ( 1+ 2^6) )
\cap 
((U_5 \setminus \{26 - ( 1 + 2^6)\}) + \calV)
= \emptyset,$$
$$
(\{ 25\} - ( 1+ 2^6) )
\cap 
(26 - ( 1 + 2^6) + (\calV\setminus \{-1\}))
= \emptyset.$$
Consequently, 
Equations
\eqref{Eqn:VContains1}, 
\eqref{Eqn:VContainsMinus1} 
hold. 

By considering the representation of the integer $-4$ as a sum of an element of $U$ and an element of $\calV$, Equation \eqref{Eqn:VContainsMinus2} follows. 

Using Equations \eqref{Eqn:RightHalfNotObtainedUV}, \eqref{Eqn:RightHalfNotObtainedLHS1ptUVbis}, we obtain 
$$
\{-6\}
\cap 
\left( 
(U_4 + (\calV\setminus \{2^3\}) 
\cup 
(\cup_{m\geq 5} U_m + \calV)
\right)
= 
\emptyset.
$$
Then Equation \eqref{Eqn:VContainsMinus4} follows from considering the representation of $-6$. 

Now, we prove that Equation \eqref{Eqn:VContainsNegativePowersOf2} holds. 
Note that the inclusions 
\begin{align*}
\cup_{0 \leq m \leq n-2} U_m + v
& \subseteq
\bbZ_{\geq 1 - (1  +2^{n-1})} + v \\
& \subseteq
\bbZ_{\geq 1 - (1  +2^{n-1}) + v} \\
& \subseteq
\bbZ_{\geq 1 - (1  +2^{n-1}) - 2^{n-1}} \\
& \subseteq
\bbZ_{\geq 1 - (1  +2^{n})} 
\end{align*}
hold for $n\geq 6$ and for any $v\in \calV$ with $v\geq -2^{n-1}$.
The inclusions 
\begin{align*}
\cup_{0 \leq m \leq n-2} U_m + v
& \subseteq
\bbZ_{\leq -1} + v \\
& \subseteq
\bbZ_{\leq -1 + v} \\
& \subseteq
\bbZ_{\leq -1 -2^{n+1}} 
\end{align*}
hold for $n\geq 6$ and for any $v\in \calV$ with $v\leq -2^{n+1}$.
The inclusions 
\begin{align*}
\cup_{0 \leq m \leq n-2} U_m - 2^n
& \subseteq
(\cup_{3 \leq m \leq n-2} \calI_m - 2^n)
\cup
(\{-2, -1\} - 2^n) \\
& \subseteq
(\bbZ_{\leq -9} - 2^n)
\cup
(\{-1, 0\} - (1 + 2^n)) \\
& \subseteq
(\bbZ_{\leq -8 - (1 + 2^n)})
\cup
(\{-1, 0\} - (1 + 2^n)) \\
& =
(\bbZ_{\leq 2^n-8 - (1 + 2^{n+1})})
\cup
(\{2^n-1,2^n\} - (1 + 2^{n+1}))
\end{align*}
hold for $n\geq 6$.
These inclusions prove that for $n\geq 6$, the third largest element of $\calI_n$ does not belong to $\cup_{0 \leq m \leq n-2} U_m + \calV$, i.e., 
$$2^n - 2 - (1 + 2^{n+1}) \notin \cup_{0 \leq m \leq  n-2} U_m + \calV.$$
Combining the above with Equations \eqref{Eqn:RightHalfNotObtainedUV}, \eqref{Eqn:RightHalfNotObtainedLHS1ptUVbis}, we obtain 
$$
2^n - 2  - (1 + 2^{n+1}) \notin 
\left(\cup_{m \geq 0, m \neq n-1, n, n+1} U_m + \calV\right) 
\cup 
(U_{n+1} + (\calV\setminus\{2^n\}))
\quad 
\text{ for } n\geq 6.
$$
Since the inclusions 
\begin{align*}
U_{n+1} + 2^n
& \subseteq 
\bbZ_{\leq 2^n + 2^{n-1} + 2^{n-3} - (1 + 2^{n+2})} + 2^n \\
& \subseteq 
\bbZ_{\leq 2^n + 2^{n-1} + 2^{n-3} - (1 + 2^{n+2}) + 2^n }\\
& \subseteq 
\bbZ_{\leq 2^{n-1} + 2^{n-3} - (1 + 2^{n+1}) }\\
& \subseteq 
\bbZ_{\leq 2^{n-1} + 2^{n-1}-3 - (1 + 2^{n+1}) }\\
& \subseteq 
\bbZ_{\leq 2^n-3 - (1 + 2^{n+1}) }
\end{align*}
hold for $n\geq 6$, it follows that 
$$
2^n - 2  - (1 + 2^{n+1}) \notin 
\left(\cup_{m \geq 0, m \neq n-1, n} U_m + \calV\right) 
\quad 
\text{ for } n\geq 6.
$$
The inclusions 
\begin{align*}
U_n  + v 
& \subseteq 
\bbZ_{\geq 1 - (1 + 2^{n+1})} + v\\
& \subseteq 
\bbZ_{\geq 1 - (1 + 2^{n+1}) + v}\\
& \subseteq 
\bbZ_{\geq 1 - (1 + 2^{n+1}) + 2^n}\\
& = 
\bbZ_{\geq 2^n + 1 - (1 + 2^{n+1})}
\end{align*}
hold for any $n\geq 5$ and for any $v\in \calV$ with $v \geq 2^n$. 
The inclusions 
\begin{align*}
U_n  + 2^{n-1}
& \subseteq 
\left(
\bbZ_{\leq 2^{n-3} + 2^{n-3} - ( 1 + 2^{n+1})} 
\cup
\bbZ_{\geq 2^{n-1} + 2^{n-2} + 1 - (1 + 2^{n+1})}
\right)
+ 2^{n-1} \\
& = 
\bbZ_{\leq 2^{n-3} + 2^{n-3} + 2^{n-1}- ( 1 + 2^{n+1})} 
\cup
\bbZ_{\geq 2^{n-1} + 2^{n-2} + 1 + 2^{n-1}- (1 + 2^{n+1})} \\
& = 
\bbZ_{\leq 2^{n-2} + 2^{n-1}- ( 1 + 2^{n+1})} 
\cup
\bbZ_{\geq 2^n +  2^{n-2} + 1 - (1 + 2^{n+1})} \\
& = 
\bbZ_{\leq 2^n - 2^{n-2}- ( 1 + 2^{n+1})} 
\cup
\bbZ_{\geq  2^n  +2^{n-2} + 1 - (1 + 2^{n+1})} \\
& \subseteq  
\bbZ_{\leq 2^n - 3 - ( 1 + 2^{n+1})} 
\cup
\bbZ_{\geq  2^n + 2^{n-2} + 1 - (1 + 2^{n+1})}
\end{align*}
hold for any $n\geq 6$.
The inclusions 
\begin{align*}
U_n  + 2^{n-2}
& \subseteq 
\left(
\bbZ_{\leq 2^{n-3} + 2^{n-3} - ( 1 + 2^{n+1})} 
\cup
\bbZ_{\geq 2^{n-1} + 2^{n-2} + 1 - (1 + 2^{n+1})}
\right)
+ 2^{n-2} \\
& = 
\bbZ_{\leq 2^{n-3} + 2^{n-3} + 2^{n-2}- ( 1 + 2^{n+1})} 
\cup
\bbZ_{\geq 2^{n-1} + 2^{n-2} + 1 + 2^{n-2}- (1 + 2^{n+1})} \\
& = 
\bbZ_{\leq  2^{n-1}- ( 1 + 2^{n+1})} 
\cup
\bbZ_{\geq 2^n +1- (1 + 2^{n+1})}
\end{align*}
hold for any $n\geq 6$. 
The inclusions 
\begin{align*}
U_n  + v 
& \subseteq 
\bbZ_{\leq 2^{n-1} + 2^{n-2} + 2^{n-4} - (1 + 2^{n+1})} + v\\
& \subseteq 
\bbZ_{\leq 2^{n-1} + 2^{n-2} + 2^{n-4} - (1 + 2^{n+1}) + v}\\
& \subseteq 
\bbZ_{\leq 2^{n-1} + 2^{n-2} + 2^{n-4} - (1 + 2^{n+1}) + 2^{n-3}}\\
& = 
\bbZ_{\leq 2^{n-1} + 2^{n-2} +2^{n-3} +  2^{n-4} - (1 + 2^{n+1}) }\\
& = 
\bbZ_{\leq 2^n -   2^{n-4} - (1 + 2^{n+1}) }\\
& \subseteq 
\bbZ_{\leq 2^n -   3 - (1 + 2^{n+1}) }\\
\end{align*}
hold for any $n\geq 6$ and for any $v\in \calV$ with $v \leq 2^{n-3}$. 
Consequently, we obtain 
$$
2^n - 2  - (1 + 2^{n+1}) \notin 
\left(\cup_{m \geq 0, m \neq n-1} U_m + \calV\right) 
\quad 
\text{ for } n\geq 6.
$$
The inclusions 
\begin{align*}
U_{n-1} + v 
& \subseteq 
\bbZ_{\geq 2^{n-4} + 1 - (1 + 2^n)} +v\\
& \subseteq 
\bbZ_{\geq 2^{n-4} + 1 - (1 + 2^n) - 2^{n-4} }\\
& = 
\bbZ_{\geq 1 - (1 + 2^n) }\\
& = 
\bbZ_{\geq 2^n + 1 - (1 + 2^{n+1}) }
\end{align*}
hold for $n\geq 6$ and for any $v\in \calV$ with $v \geq -2^{n-4}$. 
The inclusions 
\begin{align*}
U_{n-1} - 2^{n-2}  
& \subseteq 
(
\bbZ_{\leq 2^{n-4} + 2^{n-4} - (1 + 2^n) } 
\cup 
\bbZ_{\geq 2^{n-2} + 2^{n-3} + 1 - (1 + 2^n)} 
) 
- 2^{n-2} \\
& = 
(
\bbZ_{\leq 2^{n-3} - (1 + 2^n) } 
\cup 
\bbZ_{\geq 2^{n-2} + 2^{n-3} + 1 - (1 + 2^n)} 
) 
- 2^{n-2} \\
& = 
\bbZ_{\leq 2^{n-3} - (1 + 2^n) - 2^{n-2}} 
\cup 
\bbZ_{\geq 2^{n-2} + 2^{n-3} + 1 - (1 + 2^n)- 2^{n-2}}  \\
& = 
\bbZ_{\leq 2^{n-3} - (1 + 2^n) - 2^{n-2}} 
\cup 
\bbZ_{\geq 2^{n-3} + 1 - (1 + 2^n)}  \\
& = 
\bbZ_{\leq 2^{n-1} + 2^{n-2} + 2^{n-3} - (1 + 2^{n+1})} 
\cup 
\bbZ_{\geq 2^{n-3} + 1 - (1 + 2^n)}  \\
& \subseteq 
\bbZ_{\leq 2^{n-1} + 2^{n-2} + 2^{n-2} - 4 - (1 + 2^{n+1})} 
\cup 
\bbZ_{\geq 2^{n-3} + 1 - (1 + 2^n)}  \\
& = 
\bbZ_{\leq 2^n - 4 - (1 + 2^{n+1})} 
\cup 
\bbZ_{\geq 2^n + 2^{n-3} + 1 - (1 + 2^{n+1})}
\end{align*}
hold for $n\geq 6$.
The inclusions 
\begin{align*}
U_{n-1} + v 
& \subseteq 
\bbZ_{\leq 2^{n-2} + 2^{n-3} + 2^{n-5} - (1 + 2^n)} +v\\
& = 
\bbZ_{\leq 2^{n-2} + 2^{n-3} + 2^{n-5} - (1 + 2^n) +v}\\
& \subseteq 
\bbZ_{\leq 2^{n-2} + 2^{n-3} + 2^{n-5} - (1 + 2^n) - 2^{n-1}}\\
& = 
\bbZ_{\leq 2^{n-1} + 2^{n-2} + 2^{n-3} + 2^{n-5} - (1 + 2^{n+1})}\\
& \subseteq  
\bbZ_{\leq 2^{n-1} + 2^{n-2} + 2^{n-3} + 2^{n-3} -4 - (1 + 2^{n+1})}\\
& =  
\bbZ_{\leq 2^{n} -4 - (1 + 2^{n+1})}
\end{align*}
hold for $n\geq 6$ and for any $v\in \calV$ with $v \leq -2^{n-1}$. 
This yields
$$
2^n - 2  - (1 + 2^{n+1}) \notin 
\left(\cup_{m \geq 0} U_m + \calV\right) 
\cup 
(U_{n-1} + (\calV\setminus \{ - 2^{n-3} \}))
\quad 
\text{ for } n\geq 6.
$$ 
So Equation \eqref{Eqn:VContainsNegativePowersOf2} follows. 
This establishes all of the nine claims, and hence $\calV$ is a minimal complement of $U$. 
\end{proof}

\begin{theorem}
\label{Thm:CoMinU}
Let $U = \{ -u_1, - u_2, -u_3, \cdots\}$ where $u_1 < u_2 < u_3< \cdots$. Define $U_1 = U$ and for each positive integer $i\geq 1$, define 
$$U_{i+1} 
: = 
\begin{cases}
U_i \setminus \{-u_i\} & \text{ if $U_i \setminus \{-u_i\}$ is a complement to $\calV$,}\\
U_i & \text{ otherwise.}
\end{cases}
$$
Let $\calU$ denote the subset $\cap _{i \geq 1} U_i$ of $\bbZ$. The subsets $\calU$ and $\calV$ of $\bbZ$ form a co-minimal pair.
\end{theorem}

\begin{proof}
Note that for any $n\geq 4$ and any $m\geq n$, the inclusions
\begin{align*}
U_m + v
& \subseteq 
\bbZ_{\leq 2^m - (1 + 2^{m+1})} + v\\
& = 
\bbZ_{\leq 2^m - (1 + 2^{m+1}) + v}\\
& \subseteq 
\bbZ_{\leq 2^m - (1 + 2^{m+1}) + 2^{m-1}}\\
& = 
\bbZ_{\leq - (1 + 2^{m-1}) }\\
& \subseteq 
\bbZ_{\leq - 2^{m-3} -1}\\
& \subseteq 
\bbZ_{\leq - 2^{n-3}-1 }\\
\end{align*}
hold for any $v\leq 2^{m-1}$, the inclusions
\begin{align*}
U_m + 2^m
& \subseteq 
\bbZ_{\leq 2^{m-1} + 2^{m-2} + 2^{m-4} - (1 + 2^{m+1})} + 2^m \\
& = 
\bbZ_{\leq 2^{m-1} + 2^{m-2} + 2^{m-4} - (1 + 2^{m+1}) + 2^m} \\
& \subseteq  
\bbZ_{\leq 2^{m-1} + 2^{m-2} + 2^{m-3} - (1 + 2^{m+1}) + 2^m} \\
& =  
\bbZ_{\leq 2^{m-1} + 2^{m-2} + 2^{m-3} + 2^{m-3} - 2^{m-3} -  (1 + 2^{m+1}) + 2^m} \\
& =  
\bbZ_{\leq - 2^{m-3} -1 } \\
& \subseteq 
\bbZ_{\leq - 2^{n-3} -1} 
\end{align*}
hold, the inclusions 
\begin{align*}
U_m + v
& \subseteq 
\bbZ_{\geq 2^{m-3} + 1 - (1 + 2^{m+1})} + v \\
& = 
\bbZ_{\geq 2^{m-3} + 1 - (1 + 2^{m+1}) + v} \\
& \subseteq 
\bbZ_{\geq 2^{m-3} + 1 - (1 + 2^{m+1}) + 2^{m+1}} \\
& = 
\bbZ_{\geq 2^{m-3}} \\
& \subseteq 
\bbZ_{\geq 2^{n-3}}
\end{align*}
hold for any $v\geq 2^{m+1}$. 
It follows that 
\begin{equation}
\label{Eqn:UVFiniteness}
\{- 2^{n-3}, - 2^{n-3} +1, - 2^{n-3} + 2, \cdots, -2, -1, 0, 1, 2, 3, \cdots, 2^{n-3}-1\}
\cap
(\cup_{m\geq n} U_m + \calV )
= \emptyset
\end{equation}
for any $n\geq 4$.
Consequently, for any element $y\in \bbZ$, the equation $y = u +v$ holds for finitely many pairs $(u, v)\in U \times \calV$. 

For any $y\in \bbZ$ and for any $i\geq 1$, there exist elements $u_{y, i}\in U_i , v_{y, i}\in \calV$ such that $y = u_{y, i} + v_{y, i}$. 
It follows that for some element $u_y \in U$, the equality $u_y = u_{y, i}$ holds for infinitely many $i$ and for such integers $i$, we have $v_y = v_{y, i}$ where $v_y: = y - u_y\in \calV$. Thus $y - v_y = u_{y, i}$ holds for infinitely many $i$. Hence, for each integer $i\geq 1$, there exists an integer $m_i \geq i$ such that $y - v_y = u_{y, m_i}$, which yields $y\in v_y  + U_{m_i} \subseteq v_y + U_i$. Thus $y$ lies in $v_y +  \calU$, i.e., $y\in \calU + \calV$. 
Hence $\calU$ is an additive complement of $\calV$. It follows that $\calU$ is a minimal complement to $\calV$. 
By Theorem \ref{Thm:RequiredClustersV}, it follows that $\calV$ is a minimal complement to $\calU$. This proves that $(\calU, \calV)$ is a co-minimal pair. 
\end{proof}

\section{Co-minimal pairs in the integral lattices}

\begin{proposition}
\label{Prop:ABCD}
Let $H$ be a subgroup of an abelian group $G$. Let $(A, B)$ be a co-minimal pair in $H$, and $(C, D)$ be subsets of $G$ whose images form a co-minimal pair in $G/H$, and the restrictions of the map $G \to G/H$ on $C$ and $D$ induces bijections onto respective images. 
Then $(A+C, B+D)$ is a co-minimal pair in $G$. 
\end{proposition}

\begin{proof}
Note that $A+C + B+ D = G$. 

Let $a\in A, c\in C$ be such that $(A+C)\setminus \{a+c\}$ is a complement to $B+D$. Let $b\in B, d\in D$ be such that 
$a +b \notin (A \setminus \{a\}) \cdot  B$, $c+ d \modu H \notin (C \setminus \{c\}) \cdot D\modu H$. 
Since $a+c+b+d\in ((A+C)\setminus \{a+c\}) +B +D$, it follows that 
$a+c+b+d = a'+c'+b'+d'$ for some $a'\in A, c'\in C, b'\in B, d'\in D$ with $a'+c'\notin (A+C) \setminus \{a+c\}$. 
This implies $c+d \modu H = c'+d'\modu H$, and hence $c \equiv c'  \modu H, d \equiv d' \modu H$, and thus $c = c', d = d'$, which yields $a+b = a'+b'$. Since $a + b \notin (A \setminus \{a\}) +  B$, we obtain $a = a', b = b'$. This contradicts the fact that $a'+c'\notin (A+C) \setminus \{a+c\}$. Hence $A+C$ is a minimal complement to $B+D$. Similarly, it follows that $B+D$ is a minimal complement to $A+C$. 
\end{proof}

As an application of the above result, we obtain the following result. 

\begin{corollary}
Let $\calS = \{s_1, s_2, \cdots \}$. For any two sequences $\{x_n\}_{n\geq 1}$ and $\{y_n\}_{n\geq 1}$, each of 
$$
(
\{
(x_i + s_j , s_i)\,|\, i, j \geq 1
\}, 
\{
(y_u + 2^{v-1} , 2^{u-1})\,|\, u, v \geq 1
\}),$$
$$
(
\{
(x_i + 2^{j-1} , s_i)\,|\, i, j \geq 1
\}, 
\{
(y_u + s_v , 2^{u-1})\,|\, u,v \geq 1
\})$$
is a co-minimal pair in $\bbZ^2$. 
\end{corollary}

\begin{proof}

For any two sequences $\{x_n\}_{n\geq 1}$ and $\{y_n\}_{n\geq 1}$, the subsets 
$$(\calS,0) + \{(x_1, s_1), (x_2, s_2) , (x_3, s_3), \cdots \} 
=
\{
(x_i + s_j , +s_i)\,|\, i, j \geq 1
\}, $$
$$(\calT,0) + \{(y_1, 1), (y_2, 2) , (y_3, 2^2), \cdots \} 
=
\{
(y_u + 2^{v-1} , 2^{u-1})\,|\, u,v \geq 1
\}$$
of $\bbZ^2$ form a co-minimal pair in $\bbZ^2$, and the subsets 
$$(\calT,0) + \{(x_1, s_1), (x_2, s_2) , (x_3, s_3), \cdots \} 
=
\{
(x_i + 2^{j-1} , s_i)\,|\, i, j \geq 1
\}, $$
$$(\calS,0) + \{(y_1, 1), (y_2, 2) , (y_3, 2^2), \cdots \} 
=
\{
(y_u + s_v , 2^{u-1})\,|\, u,v \geq 1
\}$$
of $\bbZ^2$ form a co-minimal pair in $\bbZ^2$ by Proposition \ref{Prop:ABCD}.

\end{proof}

\begin{theorem}
\label{Thm:AExists}
Let $\sigma$ be an automorphism of $\bbZ^{n}$ such that there exists an increasing chain of subgroups
$$0 = M_0 \subsetneq M_1 \subsetneq \cdots \subsetneq M_r = \bbZ^{n}$$
of $\bbZ^{n}$ such that each of them is stable under the action of $\sigma$ and the successive quotients $M_i/M_{i-1}$ are free of rank at most two and for any such quotient, the restriction of $\sigma$ to it is the identity map if the quotient is of rank one, or conjugate to some element of $\gln_2(\bbZ)$ having exactly two nonzero entries, i.e., some of 
\begin{equation}
\label{Eqn:8Matrix}
\begin{pmatrix}
1 & 0 \\
0 & 1
\end{pmatrix}
,
\begin{pmatrix}
1 & 0 \\
0 & -1
\end{pmatrix}
,
\begin{pmatrix}
-1 & 0 \\
0 & 1
\end{pmatrix}
,
\begin{pmatrix}
-1 & 0 \\
0 & -1
\end{pmatrix}
,
\begin{pmatrix}
0 & 1 \\
1 & 0 
\end{pmatrix}
,
\begin{pmatrix}
0 & -1 \\
1 & 0 
\end{pmatrix}
,
\begin{pmatrix}
0 & 1 \\
-1 & 0 
\end{pmatrix}
,
\begin{pmatrix}
0 & -1 \\
-1 & 0 
\end{pmatrix}
\footnote{
Note that the above eight matrices form the subgroup 
$$
\left
\{
\begin{pmatrix}
a & b \\
c & d
\end{pmatrix}
\in \gln_2(\bbZ)\,|\, \text{ exactly two of $a, b, c, d$ are equal to zero}
\right
\}
$$
of $\gln_2(\bbZ)$.
}
\end{equation}
if the quotient is of rank two. 
Then there exists a subset $A$ of $\bbZ^{n}$ such that $(A, \sigma(A))$ is a co-minimal pair in $\bbZ^n$. 
\end{theorem}

\begin{proof}
If $n = 1$, then $\sigma$ is the identity map, and thus in this case, $A$ can taken to be the subset $W$ of $\bbZ$ as in \cite[Proposition 3]{Kwon}. 

If $n = 2$, then one of the following conditions hold. 
\begin{enumerate}
\item 
$M_0$ is a subgroup of $\bbZ^2$ of rank one and $\sigma$ is conjugate to 
$$
\begin{pmatrix}
1 & * \\
0 & 1
\end{pmatrix},
$$
\item $M_0 = \bbZ^2$ and $\sigma$ is conjugate to one of the matrices in Equation \eqref{Eqn:8Matrix}.
\end{enumerate}
For simplicity, we will assume that $\sigma$ is equal to the above matrix, or equal to one of the matrices in Equation \eqref{Eqn:8Matrix}. In the first case, we can take $A = W \times W$ (by Proposition \ref{Prop:ABCD}) and in the second case, we can take $A$ to be 
\begin{itemize}
\item $W \times W$, 
\item $\{(x, 0)\,|\, x\in \bbZ\} \cup \{(0, y)\,|\, y\in \bbZ_{\geq 1}\}$, 
\item $\{(x, 0)\,|\, x\in \bbZ_{\geq 1}\} \cup \{(0, y)\,|\, y\in \bbZ\}$, 
\item $\{(x, 0)\,|\, x\in \bbZ\} \cup \{(0, y)\,|\, y\in \bbZ_{\geq 1}\}$, 
\item $\calS \times \calT$,
\item $\calV \times \calU$, 
\item $\calU \times \calV$, 
\item $(-\calS) \times \calT$
\end{itemize}
according as 
$\sigma$ is equal to the matrices in Equation \eqref{Eqn:8Matrix}.
Thus Theorem \ref{Thm:AExists} holds for $n= 1, 2$. 

Let $n\geq 3$ be an integer such that Theorem \ref{Thm:AExists} holds for free abelian groups of rank $<n$. 
Let $M'$ be a subgroup of $\bbZ^n$ such that $M'$ surjects onto $M_r /M_{r-1}$.
Denote the restriction of $\sigma$ to $M_r/M_{r-1}$ by $\bar \sigma$. 
Since the result holds for $n=1, 2$ and $M_r/M_{r-1}$ has rank at most two, it follows that there is a subset $C$ of $M_r/M_{r-1}$ such that $(C, \bar \sigma (C))$ is a co-minimal pair in $M_r/M_{r-1}$. 
Let $C'$ be a subset of $M'$ such that the map $M \to M_r/M_{r-1}$ yields a bijection between $C'$ and $C$. 
Note that $M_{r-1}$ is stable under the action of $\sigma$. By the induction hypothesis, there is a subset $A$ of $M_{r-1}$ such that $(A, \sigma(A))$ is a co-minimal pair $M_{r-1}$. 
By Proposition \ref{Prop:ABCD}, $(A + C', \sigma (A) + \sigma(C'))$ is a co-minimal pair in $M_r = \bbZ^n$. Thus Theorem \ref{Thm:AExists} follows. 
\end{proof}

\begin{theorem}
\label{Thm:AExistsQuadrant}
Let $\sigma$ be an automorphism of $\bbZ^{2d}$ such that its matrix with respect to the standard basis of $\bbZ^{2d}$ has a block upper triangular form having the matrices 
$$
\begin{pmatrix}
0 & 1 \\
1 & 0 
\end{pmatrix}
,
\begin{pmatrix}
0 & -1 \\
-1 & 0 
\end{pmatrix}
$$
along the diagonal. 
Then there exists a subset $A$ of $\bbZ^{2d}$ contained in a quadrant such that $(A, \sigma(A))$ is a co-minimal pair. 
\end{theorem}

\begin{proof}
If $d = 1$, then $A$ can be taken to be $\calS \times \calT, (-\calS) \times \calT$. 
Let $d\geq 2$ be an integer and assume that the result holds for free groups of rank $2(d-1)$. 
Let $\sigma$ be an automorphism of $\bbZ^{2d}$ satisfying the given condition. 
Let $M$ denote the subgroup of $\bbZ^{2d}$ generated by $e_1, \cdots, e_{2(d-1)}$ and $M'$ denote the subgroup of $\bbZ^{2d}$ generated by $e_{2d-1}, e_{2d}$. Note that $M'$ surjects onto $\bbZ^{2d}/M$. Note that $M$ is stable under the action of $\sigma$. By the induction hypothesis, there is a subset $A$ of $M$ contained in a quadrant such that $(A, \sigma(A))$ is a co-minimal pair in $M$. 
Since the result holds for $d = 1$, it follows that there exists a subset $C$ of $M'$ contained in a quadrant such that the images of $C$ and $\sigma(C)$ in $\bbZ^{2d}/M$ form a co-minimal pair. By Proposition \ref{Prop:ABCD}, $(A+ C, \sigma(A) + \sigma(C))$ is a co-minimal pair in $\bbZ^{2d}$. Since $A+C$ is contained in a quadrant, the result follows by induction. 
\end{proof}

\begin{proof}
[Proof of Theorem \ref{Thm:Z2d}]
The first part follows from Theorem \ref{Thm:AExists}. 

For $d = 1$, the group $\bbZ^{2d}$ admits infinitely many automorphisms of the form 
$$
\begin{pmatrix}
1 & * \\
0 & 1
\end{pmatrix}
$$
with $*\in \bbZ$. So by Theorem \ref{Thm:AExists}, for each such automorphism $\sigma$ of $\bbZ^{2d}$, there is a subset $A$ of $\bbZ^{2d}$ such that $(A, \sigma(A))$ is a co-minimal pair in $\bbZ^{2d}$. 
To establish the second part for $d\geq 2$, note that there are infinitely many automorphisms of $\bbZ^{2d}$ which are block upper triangular where the blocks are of size $2\times 2$ and the matrices lying along the diagonal blocks are equal to some of 
$$
\begin{pmatrix}
1 & 0 \\
0 & 1
\end{pmatrix}
,
\begin{pmatrix}
1 & 0 \\
0 & -1
\end{pmatrix}
,
\begin{pmatrix}
-1 & 0 \\
0 & 1
\end{pmatrix}
,
\begin{pmatrix}
-1 & 0 \\
0 & -1
\end{pmatrix}
,
\begin{pmatrix}
0 & 1 \\
1 & 0 
\end{pmatrix}
,
\begin{pmatrix}
0 & -1 \\
1 & 0 
\end{pmatrix}
,
\begin{pmatrix}
0 & 1 \\
-1 & 0 
\end{pmatrix}
,
\begin{pmatrix}
0 & -1 \\
-1 & 0 
\end{pmatrix},
$$
then by Theorem \ref{Thm:AExists}, for each such automorphism $\sigma$ of $\bbZ^{2d}$, there is a subset $A$ of $\bbZ^{2d}$ such that $(A, \sigma(A))$ is a co-minimal pair in $\bbZ^{2d}$. 
\end{proof}

\section{Acknowledgements}
The first author would like to thank the Department of Mathematics at the Technion where a part of the work was carried out. The second author would like to acknowledge the Initiation Grant from the Indian Institute of Science Education and Research Bhopal, and the INSPIRE Faculty Award from the Department of Science and Technology, Government of India.


\def\cprime{$'$} \def\Dbar{\leavevmode\lower.6ex\hbox to 0pt{\hskip-.23ex
  \accent"16\hss}D} \def\cfac#1{\ifmmode\setbox7\hbox{$\accent"5E#1$}\else
  \setbox7\hbox{\accent"5E#1}\penalty 10000\relax\fi\raise 1\ht7
  \hbox{\lower1.15ex\hbox to 1\wd7{\hss\accent"13\hss}}\penalty 10000
  \hskip-1\wd7\penalty 10000\box7}
  \def\cftil#1{\ifmmode\setbox7\hbox{$\accent"5E#1$}\else
  \setbox7\hbox{\accent"5E#1}\penalty 10000\relax\fi\raise 1\ht7
  \hbox{\lower1.15ex\hbox to 1\wd7{\hss\accent"7E\hss}}\penalty 10000
  \hskip-1\wd7\penalty 10000\box7}
  \def\polhk#1{\setbox0=\hbox{#1}{\ooalign{\hidewidth
  \lower1.5ex\hbox{`}\hidewidth\crcr\unhbox0}}}
\providecommand{\bysame}{\leavevmode\hbox to3em{\hrulefill}\thinspace}
\providecommand{\MR}{\relax\ifhmode\unskip\space\fi MR }
\providecommand{\MRhref}[2]{%
  \href{http://www.ams.org/mathscinet-getitem?mr=#1}{#2}
}
\providecommand{\href}[2]{#2}

\end{document}